\documentclass[a4paper,11pt]{amsart}

\usepackage{amsthm,amsmath,amssymb,amsopn}
\usepackage{graphicx,epstopdf,epsfig,subfig}
\usepackage[colorlinks,linktocpage,linkcolor=black]{hyperref}
\usepackage{url,color}
\usepackage{booktabs}
\usepackage{colortbl,multirow}
\usepackage{appendix}
\numberwithin{equation}{section}
\numberwithin{table}{section}
\numberwithin{figure}{section}

\newtheorem{lemma}{Lemma}[section]
\newtheorem{theorem}{Theorem}[section]

\newtheorem{assumption}{Assumption}[section]

\theoremstyle{definition}

\newcommand\bm{\boldsymbol}
\newcommand{\lj}{[ \hspace{-2pt} [}
\newcommand{\rj}{] \hspace{-2pt} ]}
\newcommand{\mb}[1]{\mathbb{#1}}
\newcommand{\mc}[1]{\mathcal{#1}}
\newcommand{\abs}[1]{\left\lvert#1\right\rvert}
\newcommand{\nm}[2]{\|\,#1\,\|_{#2}}
\newcommand{\snm}[2]{\abs{\,#1\,}_{#2}}
\newcommand{\Lr}[1]{\left(#1\right)}
\newcommand{\jump}[1]{[\![#1]\!]}

\newcommand{\set}[2]{\{\,#1\,\mid\,#2\}}

\newcommand{\wt}[1]{\widetilde{#1}}
\def\MTh{\mc{T}_h}
\def\Lam{\Lambda}
\def\na{\nabla}
\def\Om{\Omega}
\def\pa{\partial}
\DeclareMathOperator{\argmin}{arg min}

\def\dx{\,\mathrm{d}x}
\def\ds{\mathrm{d}s}

\begin{document}

\markboth{R.~LI, P.B.~MING, Z.~SUN,F.~YANG AND Z.~YANG}
{Template for Journal Of Computational Mathematics}

\title{A Discontinuous Galerkin Method by Patch Reconstruction for
  Biharmonic Problem}
\author[R. Li]{Ruo Li}
\address{CAPT, LMAM and School of Mathematical Sciences, Peking
University, Beijing 100871, P. R. China} \email{rli@math.pku.edu.cn}

\author[P.-B. Ming]{Pingbing Ming}
\address{State Key Laboratory of Scientific and Engineering Computing,
Academy of Mathematics and Systems Science, Chinese Academy of Sciences, No.55, Zhong-Guan-Cun East Road,
Beijing, 100190, China}
\address{School of Mathematical Sciences, University of Chinese Academy of Sciences, No. 19A, Yu-Quan Road, Beijing, 100049, China}\email{mpb@lsec.cc.ac.cn}

\author[Z.-Y. Sun]{Zhiyuan Sun}
\address{School of Mathematical Sciences, Peking University, Beijing
  100871, P. R. China} \email{zysun@math.pku.edu.cn}

\author[F.-Y. Yang]{Fanyi Yang}
\address{School of Mathematical Sciences, Peking University, Beijing
  100871, P. R. China} \email{yangfanyi@pku.edu.cn}

\author[Z.-J. Yang]{Zhijian Yang}
\address{School of Mathematics and Statistics, Wuhan University, Wuhan
  430072, P. R. China} \email{zjyang.math@whu.edu.cn}

\maketitle

\begin{abstract}
We propose a new discontinuous Galerkin method based on the
least-squares patch reconstruction for the biharmonic problem. We
prove the optimal error estimate of the proposed method. The
two-dimensional and three-dimensional numerical examples are presented
to confirm the accuracy and efficiency of the method with several
boundary conditions and several types of polygon meshes and polyhedral
meshes.

\noindent\textbf{keyword:} Least-squares problem $\cdot$ Reconstructed
basis function $\cdot$ Discontinuous Galerkin method $\cdot$
Biharmonic problem

\end{abstract}

\section{Introduction}\label{sec:introduction}

The biharmonic boundary value problem is a fourth-order elliptic
problem that models the thin plate bending problem in continuum
mechanics, and describes slow flows of viscous incompressible
fluids. Finite element methods have been employed to approximate this
problem from its initial stage and by now there are many successful
finite element methods for this
problem~\cite{ZienkiewiczTaylorZhu:2015}.

Recently, the discontinuous Galerkin (DG)
method~\cite{Hansbo:2002,mozolevski2003priori, suli2007hp,
  mozolevski2007hp,
  georgoulis2008discontinuous,cockburn2009hybridizable} has been
developed to solve the biharmonic problem. The DG method employs
discontinuous basis functions that render great flexibility in the
mesh partition and also provide a suitable framework for
$p$-adaptivity. Higher order polynomials are easily implemented in DG
method, which may efficiently capture the smooth solutions.  To
achieve higher accuracy, DG method requires a large number of degrees
of freedom on a single element, which gives an extremely large linear
system~\cite{hughes2000comparison,zienkiewicz2003discontinuous}.  As a
compromise of the standard FEM and the DG method, Brenner et
al~\cite{engel2002continuous, brenner2011c} developed the so-called
$C^0$ interior penalty Galerkin method ($C^0$ IPG). This method
applies standard continuous $C^0$ Lagrange finite elements to the
interior penalty Galerkin variational formulation, which admits the
optimal error estimate with less local degree of freedoms compared
with standard DG method.

The aim of this work is to apply the patch reconstruction finite element
method proposed in~\cite{li2016discontinuous} to the biharmonic problem.
The main idea of the proposed method is to construct a piecewise
polynomial approximation space by patch reconstruction. The
approximation space is discontinuous across the element face and has
only one degree of freedom located inside each element, which is a
sub-space of the commonly used discontinuous Galerkin finite element
space. One advantage of the proposed method is the number of the unknown
is maintained under a given mesh partition with the increasing of the
order of accuracy. Moreover, the reconstruction procedure can be carried
out over any mesh such as an arbitrary polygonal mesh. Given such
reconstructed approximation space, we solve the biharmonic problem in
the framework of interior penalty discontinuous Galerkin (IPDG). Based
on the approximation properties of the approximation space established
in~\cite{li2012efficient} and~\cite{li2016discontinuous}, we analyze the
proposed method in the framework of IPDG. { The performance of the
proposed method is verified by a series of numerical tests with
different complexity, which is comparable with the $C^0$ IPG method
while the basis functions in our method should be pre-computed,
nevertheless, such basis functions may be reused.}

The article is organized as follows. In
Section~\ref{sec:approximationspace}, we demonstrate the
reconstruction procedure of the approximation space and develop the
corresponding approximation properties of { such a space.} Next,
the { IPDG} method with { such a reconstructed approximation
  space} is proposed and analyzed in Section~\ref{sec:ipdg}. In
Section~\ref{sec:numericalresults}, we test the proposed method by
several two-dimensional and three-dimensional biharmonic boundary
value problems, which include smooth solution as well as nonsmooth
solution for various boundary conditions and different types of
meshes.

Throughout this paper, the constant $C$ is a generic constant that may
change from line to line, but is independent of the mesh size $h$.
\section{Reconstruction Operator}
\label{sec:approximationspace}
Let $\Om\subset\mb{R}^d$ with $d=2, 3$ be a bounded convex domain. Let
$\mc{T}_h$ be a collection of $Ne$ polygonal elements that partition
$\Om$. We denote all interior faces of $\mc{T}_h$ as $\mc{E}_h^i$ and
the set of the boundary faces as $\mc{E}_h^b$, and then
$\mc{E}_h=\mc{E}_h^i\cup\mc{E}_h^b$ is then a collection of all
$(d-1)$-dimensional faces of all elements in $\mc{T}_h$. Let
$h_K=\text{diam}K$ and $h=\max_{K\in\mc{T}_h}h_K$. We assume that
$\mc{T}_h$ satisfies the shape-regular conditions used in~\cite{
  antonietti:2013, brezzi2005family}, which read: there exist
\begin{enumerate}
  \item[-] an integer number $N$ independent of $h$;
  \item[-] a real positive number $\sigma$ independent of $h$;
  \item[-] a compatible sub-decomposition $\widetilde{\mathcal T_h}$
    into shape-regular triangles or quadrilaterals, or mix of both
    triangles and quadrilaterals,
\end{enumerate}
such that
\begin{itemize}
  \item[(A1)] any polygon $K\in\mathcal T_h$ admits a decomposition
    $\wt{\mc{T}_h}$ formed by less than $N$ triangles;

  \item[(A2)] any triangle $T\in\wt{\mc{T}_h}$ is shape-regular in the
    sense that there exists $\sigma$ satisfying $h_K/\rho_K\leq
    \sigma$, where $\rho_K$ is the radius of the largest ball
    inscribed in $K$.
\end{itemize}

The above regularity assumptions lead to some useful consequences,
which will be extensively used in the later analysis.
\begin{enumerate}
\item[{\bf M1}]For any triangle $T\in\wt{\MTh}$, there exists
  $\rho_1\ge 1$ that depends on $N$ and $\sigma$ such that
  $h_K/h_T\le\rho_1$.

\item[{\bf M2}][{\it Agmon inequality}]\;There exists $C$ that depends
  on $N$ and $\sigma$, but independent of $h_K$ such that
\begin{equation}\label{eq:agmon}
\nm{v}{L^2(\pa K)}^2\le C\Lr{h_K^{-1}\nm{v}{L^2(K)}^2+h_K\nm{\na
    v}{L^2(K)}^2}\qquad\text{for all\quad}v\in H^1(K).
\end{equation}

\item[{\bf M3}][{\it Approximation property}]\;There exists $C$ that
  depends on $N$ and $\sigma$, but independent of $h_K$ such that for
  any $v\in H^{m+1}(K)$, there exists an approximating polynomial
  $\wt{v}_m\in\mb{P}_m(K)$ such that
\begin{equation}\label{eq:app}
\begin{aligned}
  \nm{v-\wt{v}_m}{L^2(K)}&+h_K\nm{\na(v-\wt{v}_m)}{L^2(K)}\\
  &\quad+h_K^2\nm{\na^2(v-\wt{v}_m)}{L^2(K)}
  \le Ch_K^{m+1}\snm{v}{H^{m+1}(K)}.
\end{aligned}
\end{equation}

\item[{\bf M4}][{\it Inverse inequality}]\;For any $v\in\mb{P}_m(K)$,
  there exists a constant $C$ that depends only on $N$ and $\sigma$
  such that
\begin{equation}\label{eq:inverse}
\nm{\na v}{L^2(K)}\le Cm^2/h_K\nm{v}{L^2(K)}.
\end{equation}
\end{enumerate}

Given the triangulation $\MTh$, we define the reconstruction operator
as follows. First, for each element $K$, we prescribe a point $x_K$ as
the collocation point. In particular, we may specify $x_K$ as the
barycenter of $K$, although it could be more flexible. Second, we
construct an element patch $S(K)$ that consists of $K$ itself and some
elements around $K$. The element patch can be built in two ways. The
first way is that we initialize $S(K)$ with $K$, and we add all Von
Neumann neighbours of the elements into $S(K)$ in a recursive manner
until we have collected enough large number of elements; see
Figure~\ref{fig:buildpatch} for such an example of $S(K)$ with
$\#S(K)=12$, where $\#S(K)$ the number of elements inside
$S(K)$. Another way to construct $S(K)$ has been reported
in~\cite{li2012efficient}. Denote by $\mc{I}_K$ the set of the
collocation points that belong to $S(K)$. It is clear that
$\#\mc{I}_K=\# S(K)$.
\begin{figure}
  \centering
  \includegraphics[width=0.8\textwidth]{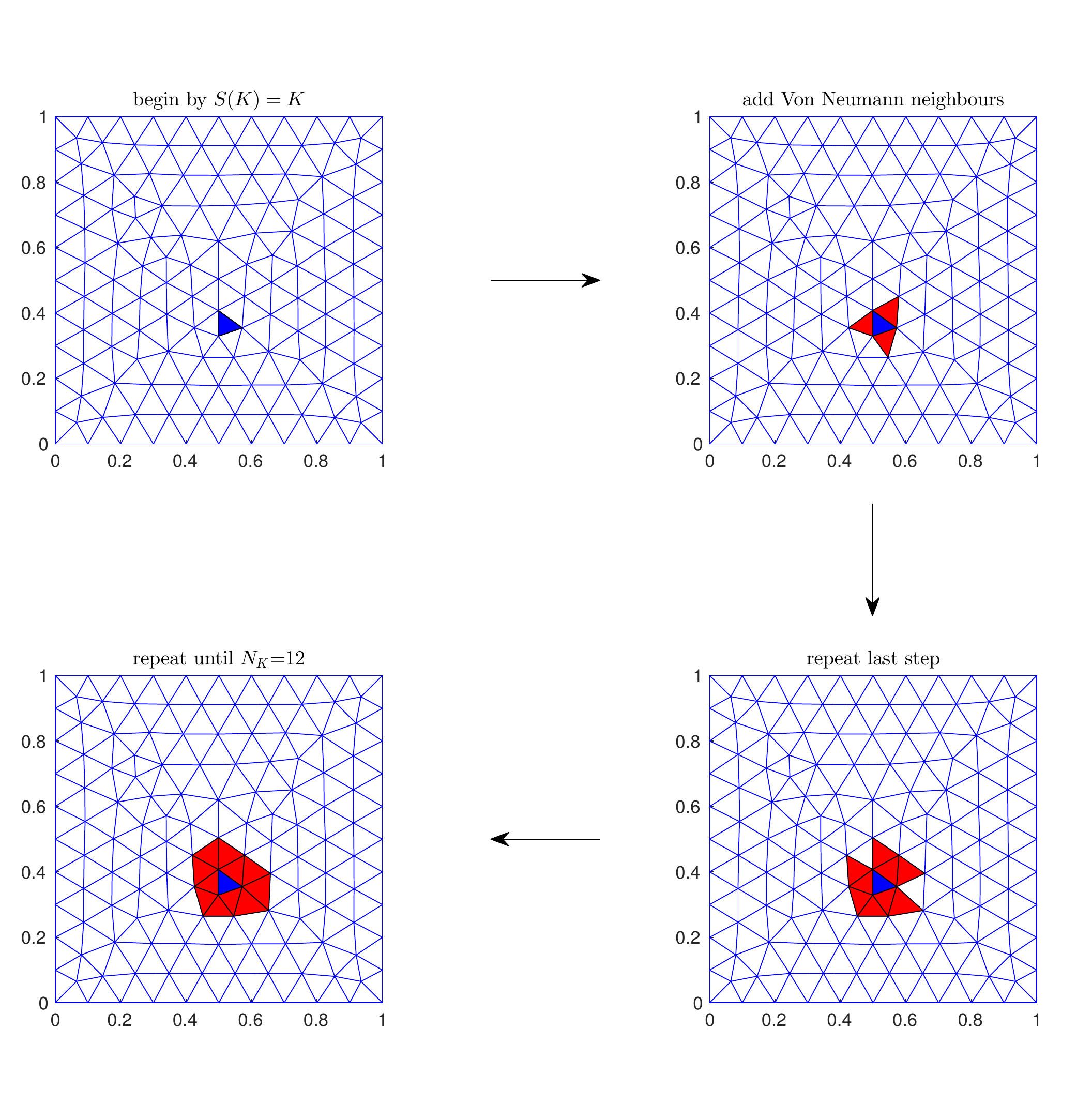}
  \caption{Build an element patch.}
  \label{fig:buildpatch}
\end{figure}

Let $U_h$ be the piecewise constant space associated with $\MTh$,
i.e.,
\[
  U_h{:}=\set{v\in L^2(\Om)}{v|_K \in \mb{P}_0(K),\ \text{for
      all\quad} K\in \MTh},
\]
where $\mb{P}_r$ is the polynomial space of degree not greater than
$r$.

For any function $g\in U_h$, we reconstruct a polynomial $\mc{R}_K g$
of degree $m$ on $S(K)$ by solving the following least-squares:
\begin{equation} \label{eq:lsproblem}
  \mathcal R_Kg= \argmin_{p\in\mb{P}_m(S(K))} \ \sum_{x\in
    \mc{I}_K}\abs{g(x)-p(x)}^2.
\end{equation}

The uniqueness condition for Problem~\eqref{eq:lsproblem} relates to
the location of the collocation points and $\# S(K)$. Following
\cite{li2012efficient, li2016discontinuous}, we make the following
assumption:
\begin{assumption}\label{as:unique}
For all $K\in\MTh$ and $p\in \mb{P}_m(S(K))$,
\begin{equation}
    p|_{\mc{I}_K}=0\quad\text{implies\quad} p|_{S(K)}\equiv0,
\end{equation}
\end{assumption}
The above assumption guarantees the uniqueness of the solution of
Problem~\eqref{eq:lsproblem} if $\# S(K)$ is greater than
$\text{dim}\,\mb{P}_m$. Hereafter, we assume that this assumption is
always valid.

From~\eqref{eq:lsproblem} we define the global reconstruction operator
$\mathcal R$ for $g\in U_h$ by restricting the polynomial $\mathcal
R_Kg$ on $K$, { \(\mc{R}g|_K = (\mc{R}_Kg)|_K\). Therefore, the
  operator $\mc R$ defines a map from $U_h$ to a discontinuous
  piecewise polynomial space with order $m$, which is denoted as $V_h
  = \mc R U_h$.

Next we investigate the behaviour of the basis functions, which are
generated by the reconstruction process. Let $\bm e_K \in \mathbb
R^{Ne}$, whose components are given by
\begin{displaymath}
  e_{K, \widetilde K} = \delta_{K, \widetilde K}, \quad \text{for
    all\quad} \widetilde K \in \mc T_h,
\end{displaymath}
where
\begin{displaymath}
  \delta_{K,\widetilde K}=\begin{cases} 1,\ K=\widetilde
  K,\\ 0,\ K\neq\widetilde K.\\
  \end{cases}
\end{displaymath}
Then we denote $\{ \lambda_K | \lambda_K = \mathcal R\bm e_K$
$\text{for all\quad} K\in\mathcal T_h \}$ as a group of basis
functions. Given $\{\lambda_K\}$ we may write the reconstruction
operator in an explicit way:
\begin{equation}
  \mathcal Rg=\sum_{K\in\mathcal T_h}g(\bm x_K)\lambda_K(\bm
  x),\quad\text{for all\quad} g\in U_h.
  \label{eq:interpexpression}
\end{equation}
The operator $\mc R$ may be defined for continuous functions posed on
$\Om$ as~\eqref{eq:interpexpression} .

The support of $\lambda_K$ can be described as:
\begin{displaymath}
  \text{supp}\lambda_K=\bigcup_{ K'\in S(K)}\overline{K'}.
\end{displaymath}
A one-dimensional example is presented in
Section~\ref{sec:numericalresults} and Figure~\ref{fig:1d_basis}.

It is worthwhile to mention that firstly the support of the basis
function may not be equal to the element patch, and vice versus;
secondly, the basis function is discontinuous, which lends itself
ideally to the DG framework.}

To state the method, we introduce some notations. Define the broken
Sobolev spaces:
\[
  H^s(\Omega,\MTh)=\set{u\in L^2(\Om)}{u|_K \in H^s(K)\;\text{for
      all\;} K\in\MTh},
\]
where $H^s(K)$ is the standard Sobolev space on element $K$ associated
with the norm $\|\cdot\|_{H^s(K)}$ and the seminorm
$\abs{\cdot}_{H^s(K)}$. The associated broken norm and seminorm are
defined respectively by
\[
  \nm{u}{H^s(\Omega,\MTh)}^2=
  \sum_{K\in\MTh}\|u\|_{H^s(K)}^2,\quad\text{and}\quad
  \abs{u}_{H^s(\Om,\MTh)}^2= \sum_{K\in\MTh}\abs{u}_{H^s(K)}^2.
\]

For two neighbouring elements $K^+, K^-$ that share a common face
$e=\partial K^+\cap\partial K^-$ with $e\in\mc{E}_h^i$, denote by
$n^+$ and $n^-$ the outward unit normal vectors on $e$, corresponding
to $\partial K^+$ and $\partial K^-$, respectively. For any function
$q\in H^1(\Omega,\MTh)$ and $v\in [H^1(\Omega,\MTh)]^d$ that may be
discontinuous across interelement boundaries, we define the
\emph{average} operator $\{\cdot\}$ and the \emph{jump} operator
$\lj\cdot\rj$ as
\[
  \{q\}=\frac12(q^++q^-), \quad\{v\}=\frac12(v^++v^-),
\]
and
\[
  \jump{q}=n^+q^++n^-q^-, \quad\jump{v}= n^+\cdot v^++n^-\cdot v^-.
\]
Here $q^+=q|_{K^+},\ v^+=v|_{K^+}$ and
$q^-=q|_{K^-},\ v^-=v|_{K^-}$. In the case $e\in\mc{E}_h^b$, there
exists an element $K$ such that $e=K\cap\partial\Omega$, the
definitions are changed to
\[
  \{q\}=q|_{K\cap\partial\Omega},\quad \jump{q}=
  nq|_{K\cap\partial\Omega},
\]
and
\[
  \{v\}=v|_{K\cap\partial\Omega},\quad \jump{v}=n\cdot v|_{K
    \cap\partial\Omega}.
\]

Following~\cite{suli2007hp}, for any $w\in H^2(\Omega,\MTh)$, we
define the energy norm $\nm{\cdot}{h}$ as
\[
 \nm{w}{h}^2=\sum_{K\in\MTh}\nm{\Delta w}{L^2(K)}^2 +
 \nm{h_e^{-3/2}\jump{w}}{L^2(\mc{E}_h)}^2 +\nm{h_e^{-1/2}\jump{\na w
 }}{L^2(\mc{E}_h)}^2,
\]
where
\[
\nm{h_e^{-3/2}\jump{w}}{L^2(\mc{E}_h)}^2{:}=\sum_{e\in\mc{E
  }_h^i}h_e^{-3}\nm{\jump{w}}{L^2(e)}^2\qquad \nm{h_e^{-1/2}\jump{\na
    w
}}{L^2(\mc{E}_h)}^2{:}=\sum_{e\in\mc{E}_h^i}h_e^{-1}\nm{\jump{\na
    w}}{L^2(e)}^2.
\]
This energy norm $\nm{w}{h}$ is equivalent to
$\nm{w}{H^2(\Omega,\MTh)}$.

{ Next, we turn to the stability estimate of the reconstruction
  operator $\mathcal R$. By~\cite{li2012efficient}, we define
  $\Lambda(m,\mc I_K)$ for any element $K \in \mc T_h$ as
\begin{equation}\label{eq:constant}
  \Lambda(m, \mathcal{I}_K) \triangleq \max_{p\in \mathbb{P}_m(S(K))}
  \frac{\max_{x\in S(K)} |p(\bm{x})|}{\max_{x\in \mathcal{I}_K}
    |p(x)|}.
\end{equation}
By Assumption~\ref{as:unique}, we may conclude that $\Lambda(m, \mc
I_K)$ is finite. With some further conditions on $S(K)$, we conclude
that $\Lambda(m, \mc I_K)$ has a uniform upper bound, which is denoted
by $\Lambda_m$. One of such condition may be found in the following
lemma.
\begin{lemma}~\cite[Lemma 3.5]{li2012efficient}
If each element patch $S(K)$ is convex and the triangulation is
quasi-uniform, then $\Lambda_m$ is uniformly bounded.
\end{lemma}

The above condition is not necessary, we refer
to~\cite{li2016discontinuous} for the discussion on the uniform upper
bound for non-convex element patch.

With the uniformly bounded $\Lam_m$, we have the following
quasi-optimal approximation estimates for the reconstruction operator
in the maximum norm.
\begin{lemma}\cite[Theorem 3.3]{li2012efficient}
  If Assumption~\ref{as:unique} holds, the stability estimate holds
  true for any $K \in \mc T_h$ and $g \in C^0(S(K))$ as
  \begin{equation}
    \|g - \mc R g\|_{L^\infty(K)} \leq C \Lambda_m \inf_{ p \in
      \mathbb P_m (S(K))} \| g - p\|_{L^\infty(S(K))},
    \label{eq:inftyestimate}
  \end{equation}
  where $C$ is independent of $h$ but depends on $\# S(K)$.
\end{lemma}

The above lemma immediately implies the quasi-optimal approximation
results in other norms.
\begin{lemma}
Let $g\in H^t(\Omega)(t\geq 2)$ and $ K\in\mathcal T_h$, there exists
$C$ independent of $h$ but depends on $\# S(K)$ such that for
$q=0,1,2$,
\begin{equation}
    \|g-\mathcal Rg\|_{H^q(K)}\leq C\Lambda_{m}
    h^{s-q}\|g\|_{H^{t}(S(K),\mathcal T_h)},
    \label{eq:reconSobleverror}
  \end{equation}
and for $q=1,2$
\begin{equation} \label{eq:recontraceinequality}
    \|\na^q(g-\mathcal R g)\|_{L^2(\partial K)}\leq C\Lambda_{m}
    h^{s-q-1/2}\|g\|_{H^t(S(K),\mathcal T_h)}.
  \end{equation}
where $s=\min(m+1,t)$.
\end{lemma}

\begin{proof}
  By the standard interpolation estimate and~\eqref{eq:inftyestimate},
  we obtain
  \[
    \begin{aligned}
      \|g - \mc R g\|_{L^2(K)} &\leq |K|^{1/2} \|g - \mc R_m
      g\|_{L^\infty(K)} \\ &\le C|K|^{1/2}\Lambda_m\inf_{ p \in
        \mathbb P^m (S(K))} \| g - p\|_{L^\infty(S(K))} \\ &\le
      C\Lambda_m h^s\|g\|_{H^t(S(K), \mc T_h)},
    \end{aligned}
  \]
  which gives~\eqref{eq:reconSobleverror} with $q=0$.
  
Let $g_m$ be the approximation polynomial in~\eqref{eq:app} for $g$,
using~\eqref{eq:inverse}, we obtain
\begin{align*}
  \nm{\na(g-\mc{R}g)}{L^2(K)}&\le\nm{\na(g-g_m)}{L^2(K)}+\nm{\na(g_m-\mc{R}g)}{L^2(K)}\\
  &\le\nm{\na(g-g_m)}{L^2(K)}+Ch_K^{-1}\nm{g_m-\mc{R}g}{L^2(K)}\\
  &\le\nm{\na(g-g_m)}{L^2(K)}+Ch_K^{-1}\nm{g-g_m}{L^2(K)}+Ch_K^{-1}\nm{g-\mc{R}g}{L^2(K)}\\
  &\le Ch_K^s\nm{g}{H^t(S(K),\mc{T}_h)},
\end{align*}
which yields~\eqref{eq:reconSobleverror} with $q=1$ by
using~\eqref{eq:reconSobleverror} with $q=0$ in the last step.

Proceeded along the same line that leads
to~\eqref{eq:reconSobleverror} with $q=1$, we
obtain~\eqref{eq:reconSobleverror} with $q=2$.

Using {\em Agmon inequality}~\eqref{eq:agmon}
and~\eqref{eq:reconSobleverror}, we
obtain~\eqref{eq:recontraceinequality}, which completes the proof.
\end{proof}

Using the above lemma, we obtain the following interpolation estimate
for the reconstruction operator in the energy norm.
\begin{theorem}\label{th:interpolationerrorDG}
  Let $g\in H^t(\Omega)$ with $t\ge 2$. There exists a constant $C$
  independent of $h$ such that
  \begin{equation}\label{eq:interpolation}
    \nm{g-\mc{R}g}{h}\le C\Lam_m h^{s-2}\|g\|_{H^t(\Omega,\MTh)},
    \end{equation}
    where $s=\min(m+1,t)$.
\end{theorem}

}
\section{Error Estimate for the Biharmonic Problem}\label{sec:ipdg}
Let us consider the biharmonic problem: find $u\in H^4(\Omega)$, such
that
\begin{equation}  \label{eq:biharmonic}
  \begin{cases}
    \Delta^2u=f,\quad &x\in\Omega,\\ u=g_D,\quad &
    x\in\partial\Omega,\\ \pa_n u=g_N,\quad & x\in\partial\Omega,
  \end{cases}
\end{equation}
where $\Delta^2u=\Delta(\Delta u)$, $n$ denotes the unit outward
normal vector to $\partial\Omega$, $f\in L^2(\Omega)$, $g_D$ and $g_N$
are assumed to be suitably smooth such that under proper conditions on
$\Omega$, the boundary value problem~\eqref{eq:biharmonic} has a
unique solution. We refer to~\cite{blum1980boundary} for precise
description of such results.

The IPDG method~\cite{georgoulis2008discontinuous} employs the
following bilinear form $B$: for any $v,w\in H^4(\Omega,\MTh)$,
  \begin{align*}
   B(v, w)&=\sum_{K\in\MTh}\int_K\Delta v\Delta
   w\dx+\sum_{e\in\mc{E}_h}\int_e (\jump{v}\cdot\{\nabla\Delta
   w\}+\jump{w}\cdot\{\nabla\Delta v\})\mathrm
   ds\\ &\quad-\sum_{e\in\mc{E}_h}\int_e(\{\Delta w\}\jump{\nabla
     v}+\{\Delta v\} \jump{\nabla
     w})\ds+\sum_{e\in\mc{E}_h}\int_e\Lr{\alpha\jump{v}\jump{w}+\beta\jump{\nabla
       v}\jump{\nabla w}}\ds.
  \end{align*}
The piecewise penalty parameters $\alpha$ and $\beta$ are nonnegative
and will be specified later on. The linear form $l$ is defined for all
$v\in H^4(\Omega,\MTh)$ as
\[
l(v)=\int_\Omega fv\dx+\sum_{e\in\mc{E}_h^b}\int_e\Lr{g_D[\pa_n\Delta
    v+\alpha v]+g_N[\beta\pa_nv-\Delta v]}\ds.
\]

The discretized problem is to find $u_h\in U_h$ such that for all
$v_h\in U_h$,
\begin{equation} \label{eq:disproblem}
 B(\mathcal Ru_h,\mathcal R v_h)=l(\mathcal R v_h).
\end{equation}

{ We begin the error estimate by defining the lifting operator
  $\mathcal L: H^4(\Omega,\mathcal T_h)\rightarrow P_h$
  \cite{georgoulis2008discontinuous}:
\begin{equation}
  \int_\Omega \mathcal L(w)r\dx=\int_{\mc E_h}(\lj w\rj\cdot\{ \nabla r\}
  -\{r\}\lj\nabla w\rj)\mathrm ds\quad\text{for all\;} r\in P_h,
  \label{eq:liftingoperator}
\end{equation}
where $P_h$ is defined by
\begin{displaymath}
  P_h\triangleq\big \{w\in L^2(\Omega)\ \big |\ w|_K\in\mathbb
  P_m(K),\ \forall K\in\mathcal T_h\big \}.
\end{displaymath}

The following lemma~\cite{georgoulis2008discontinuous} gives the
stability of the lifting operator.
\begin{lemma}
  For all $w\in H^4(\Omega,\mathcal T_h)$, the following estimate
  holds:
  \begin{displaymath}
    \|\mathcal L(w)\|_{L^2(\Omega)}^2\leq\|\sqrt\gamma\lj w\rj\|_{
      L^2(\mc E_h)}^2+\|\sqrt\delta\lj\nabla w\rj\|_{L^2(\mc E_h)}^2,
  \end{displaymath}
  for piecewise constants $\gamma, \delta$ that are defined for all
  $e\in\mc E_h$ by
  \begin{displaymath}
    \gamma|_e=\frac{C_\gamma}{h_e^3},\quad
    \delta|_e=\frac{C_\delta}{h_e}, \quad\text{for all\quad}
    e\in\mc E_h,
  \end{displaymath}
  with sufficiently large positive constants $C_\gamma$ and
  $C_\delta$.
  \label{le:liftingoperator}
\end{lemma}

With the definition of the DG energy norm and the Lemma
\ref{le:liftingoperator}, we show the bilinear form $ B(\cdot,\cdot)$
satisfies the boundedness and the stability condition.
\begin{lemma}
Let $\alpha,\beta>0$, there exists a positive constant $\Lambda$ which
is independent of mesh size $h$ such that for all $v, w\in
H^4(\Omega,\mathcal T_h)$, there holds
  \begin{equation}
    |B(v,w)|\leq \Lambda\|v\|_{h}\| w\|_{h}.\quad
  \label{eq:bilinearboundedness}
  \end{equation}
  \label{th:bilinearboundedness}
\end{lemma}

\begin{proof}
  By the lifting operator \eqref{eq:liftingoperator}, the bilinear
  form $B$ may be written as
  \begin{equation}
    \begin{aligned}
      B(v,w)=\int_{\Omega}\big (\Delta v\Delta w& + \mathcal
      L(v)\Delta w+\Delta v\mathcal L(w)\big )\dx\\ &+\int_{\mc E_h}\big
      (\alpha\lj v\rj\lj w\rj+\beta \lj\nabla v\rj\lj \nabla w\rj\big
      )\mathrm ds.\\
    \end{aligned}
    \label{eq:bilinearlifting}
  \end{equation}
  Therefore, we obtain
  \begin{displaymath}
    \begin{aligned}
      |B(v, w)|\leq&\|\Delta v\|_{L^2(\Omega)}\|\Delta
      w\|_{L^2(\Omega)}+\|\mathcal L(v)\|_{L^2(\Omega)}\|\Delta
      w\|_{L^2(\Omega)}\\ &+\|\Delta v\|_{L^2(\Omega)}\|\mathcal
      L(w)\|_{L^2(\Omega)} +\|\sqrt\alpha\lj
      v\rj\|_{L^2(\Omega)}\|\sqrt\alpha\lj
      w\rj\|_{L^2(\mc E_h)}\\ &+\|\sqrt\beta \lj \nabla v
      \rj\|_{L^2(\mc E_h)} \| \sqrt\beta \lj \nabla w\rj
      \|_{L^2(\mc E_h)}\\ \leq& \Lambda\| v \|_{h} \| w \|_{h}.
    \end{aligned}
  \end{displaymath}
\end{proof}
\begin{lemma}
  Let
 \begin{equation}\label{eq:parameter}
\alpha | _e=\mu/h_e^3\quad\text{and\quad} \beta |_e=\eta/h_e
\end{equation}
  on $e\in\mc E_h$, where $\mu$ and $\eta$ are positive constants. With
  sufficiently large $\mu$ and $\eta$, there exists a positive
  constant $\lambda$ that is independent of mesh size $h$ such that
  for all $v_h\in U_h$
  \begin{equation}
    B(\mc{R}v_h,\mc{R}v_h)\geq \lambda\|\mc{R}v_h\|_{h}^2.
    \label{eq:bilinearcoercivity}
  \end{equation}
  \label{le:bilinearcoercivity}
\end{lemma}

\begin{proof}
By the definition, we write
  \begin{displaymath}
    \begin{aligned}
      B(\mc{R}v_h,\mc{R}v_h)&=\|\Delta
      \mc{R}v_h\|_{L^2(\Omega)}^2+\|\sqrt\alpha \lj \mc{R}v_h
      \rj\|_{L^2(\mc E_h)}^2+\|\sqrt\beta \lj \nabla \mc{R}v_h \rj\|_{
        L^2(\mc E_h)}^2\\ &\hspace{100pt}+2\int_\Omega\mathcal
      L(\mc{R}v_h)\Delta \mc{R}v_h\dx.\\
    \end{aligned}
  \end{displaymath}
  Using the inequality
  \begin{displaymath}
    -2\int_\Omega\mathcal L(\mc{R}v_h)\Delta \mc{R}v_h\dx\leq
    2\|\mathcal L(\mc{R}v_h)\|_{ L^2(\Omega)}^2 + \frac12\| \Delta
    \mc{R}v_h \|_{L^2(\Omega)}^2,
  \end{displaymath}
and Lemma~\ref{le:liftingoperator}, we obtain
  \begin{displaymath}
    B(\mc{R}v_h, \mc{R}v_h)\geq\frac12\| \Delta \mc{R}v_h \|_{L^2(
      \Omega)}^2 + \|\sqrt{\alpha - 4\gamma}\lj \mc{R}v_h \rj
      \|_{L^2(\mc{E}_h)}^2 + \| \sqrt{ \beta - 4\delta}\lj \nabla
    \mc{R}v_h \rj\|_{L^2(\mc E_h)}^2,
  \end{displaymath}
  hence, the coercivity follows if $\alpha>4\gamma$ and
  $\beta>4\delta$.
\end{proof}

We are ready to prove a priori error estimates for
Problem~\eqref{eq:disproblem}.
\begin{theorem}\label{th:errorbound}
Let $u_h$ be the solution of Problem~\eqref{eq:disproblem} and let the
exact solution $u$ to the Problem~\eqref{eq:biharmonic} belong to the
broken Sobolev space $H^t(\Omega,\mathcal T_h)$ with $t\ge
4$. Furthermore, let the penalty functions $\alpha$ and $\beta$
satisfy the condition~\eqref{eq:parameter} in
Lemma~\ref{le:bilinearcoercivity}. Then
\begin{equation}\label{eq:DGerror}
    \nm{u-\mathcal R u_h}{h}\le Ch^{s-2}\nm{u}{H^t(\Om,\MTh)},
\end{equation}
where $s=\min(m+1,t)$ and $s\ge 3$.
\end{theorem}

\begin{proof}
We begin with the Galerkin orthogonality: for all $v_h\in U_h$, there
holds
\begin{equation} \label{eq:orthogonality}
 B(u-\mc{R}u_h,\mc{R}v_h)=0.
\end{equation}
Denote $w=\mc{R} u-\mc{R} u_h$, and using the above Galerkin
orthogonality, we obtain
\[
B(w,w)=B(\mc{R}u-u,w)+B(u-\mc{R}u_h,w)=B(\mc{R}u-u,w).
\]
Using~\eqref{eq:bilinearcoercivity}
and~\eqref{eq:bilinearboundedness}, we obtain
\[
\nm{\mc{R}u-\mc{R}u_h}{h}\le\dfrac{\Lambda}{\lambda}\nm{u-\mc{R}u}{h}.
\]
By the triangle inequality we immediately obtain
\[
\nm{u-\mc{R}u_h}{h}\le(1+\Lambda/\lambda)\nm{u-\mc{R}u}{h},
\]
which together with~\eqref{eq:interpolation} gives the
estimate~\eqref{eq:DGerror}.
\end{proof}

The $L^2$ error estimate may be obtained by standard duality
argument. Following~\cite{mozolevski2007hp}, we make the regularity
assumptions: the solution $\psi$ of problem
\begin{equation}
\left\{
  \begin{aligned}
    \Delta^2\psi&=u-\mc{R}
    u_h,\quad&&\text{in}\quad\Omega,\\ \psi=\dfrac{\pa\psi}{\pa
      n}&=0,\quad&&\text{on\quad}\partial\Omega,
  \end{aligned}\right.
  \label{eq:dualproblem}
\end{equation}
belongs to $H^4(\Omega)$, and there exists a positive constant $C$
that only depends on $\Omega$ such that
\begin{equation} \label{eq:ellpticregularity}
  \|\psi\|_{H^4(\Omega)}\leq C\|u-\mc{R} u_h\|_{L^2(\Omega)}.
\end{equation}

\begin{theorem}
Besides the conditions in Theorem~\ref{th:errorbound}, we assume the
elliptic regularity~\eqref{eq:ellpticregularity} holds true, then
  \begin{align}
    \|u-\mc Ru_{h}\|_{L^2(\Omega)}&\leq Ch^2\|u\|_{H^t(\Omega,\mathcal
      T_h)} ,\quad m=2, \label{eq:L2errorm=2}\\ \|u-\mc
    Ru_{h}\|_{L^2(\Omega)}&\leq Ch^{s}\|u\|_{H^t(\Omega,\mathcal
      T_h)},\quad m\geq3,\label{eq:L2errorm=3}
  \end{align}
  where $s=\min(m+1,t)$.
\end{theorem}

\begin{proof}
An integration by parts gives
\[
\|u-\mc{R}u_h\|_{L^2(\Om)}^2=B(\psi,u - \mc{R}u_h)=B(\psi-\mc{R}\psi,
u-\mc{R}u_h),
\]
where we have used the Galerkin orthogonality in the last step.
Combining the interpolation estimate, the energy
estimate~\eqref{eq:DGerror} and the regularity
estimate~\eqref{eq:ellpticregularity} gives~\eqref{eq:L2errorm=2}
and~\eqref{eq:L2errorm=3}.
\end{proof}
}

\section{Implementation and Numerical Results}
\label{sec:numericalresults}
{ In this section, we describe the implementation details of the
reconstructed space and report some numerical examples to show the
accuracy and performance of the proposed method. In all examples, we
build element patches by the first method and use a sparse direct
solver for the resulting sparse linear systems.
\subsection{Implementation}
The key point of the implementation is to calculate the basis
functions and here we present a one-dimensional example on the
interval $[0, 1]$. Consider a uniform mesh consisting of $5$ elements
$\{K_1, K_2, K_3, K_4, K_5\}$; see Figure~\ref{fig:1d_mesh}.
\begin{figure}
  \centering \includegraphics[width=10cm]{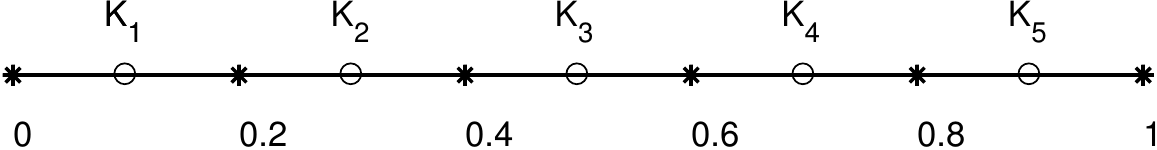}
  \caption{The uniform mesh on $[0,1]$.}
  \label{fig:1d_mesh}
\end{figure}
We choose the midpoint of each element as the collocation point to
reconstruct a piecewise linear space.  The element patches are taken
as
\begin{align*}
S(K_1)&=\{K_1,K_2,K_3\},\quad
S(K_i)=\{K_{i-1},K_i,K_{i+1}\},\ i=2,3,4,\\ S(K_5)&=\{K_3,K_4,K_5\}.
\end{align*}

The local least-squares~\eqref{eq:lsproblem} on element $K_i$ is
\begin{equation*}
  \mc R_{K_i} g= \argmin_{\{a,b\}\in\mb{R}} \ \sum_{x_{K'}\in
    \mc{I}_{K_i}}\abs{g(x_{K'})-(ax_{K'}+b)}^2.
\end{equation*}
A direct calculation gives
\begin{displaymath}
  [a, b]^T = (A^T A)^{-1} A^T q,
\end{displaymath}
where $A$ and $q$ are given as follows. For $i=1$,
\[
A = \begin{bmatrix} 1 & x_{K_1} \\ 1 & x_{K_2} \\ 1 & x_{K_3}
  \end{bmatrix} ,\quad
  q = \begin{bmatrix} g(x_{K_1}) \\ g(x_{K_2}) \\ g(x_{K_3})
  \end{bmatrix} ,
\]
and for $i=2,3,4$,
\[
  A = \begin{bmatrix} 1 & x_{K_{i-1}} \\ 1 & x_{K_i} \\ 1 &
    x_{K_{i+1}}
    \end{bmatrix},\quad
    q = \begin{bmatrix} g(x_{K_{i-1}}) \\ g(x_{K_i}) \\ g(x_{K_{i+1}})
  \end{bmatrix} ,
  \]
 and for $i=5$,
\[
  A = \begin{bmatrix} 1 & x_{K_3} \\ 1 & x_{K_4} \\ 1 & x_{K_5}
  \end{bmatrix} \quad
    q = \begin{bmatrix} g(x_{K_3}) \\ g(x_{K_4}) \\ g(x_{K_5})
  \end{bmatrix} .
  \]
The basis functions $\{ \lambda_K\}$ are plot in
Figure~\ref{fig:1d_basis}.
\begin{figure}
  \centering
  \begin{tabular}{c}
    \includegraphics[scale = 1]{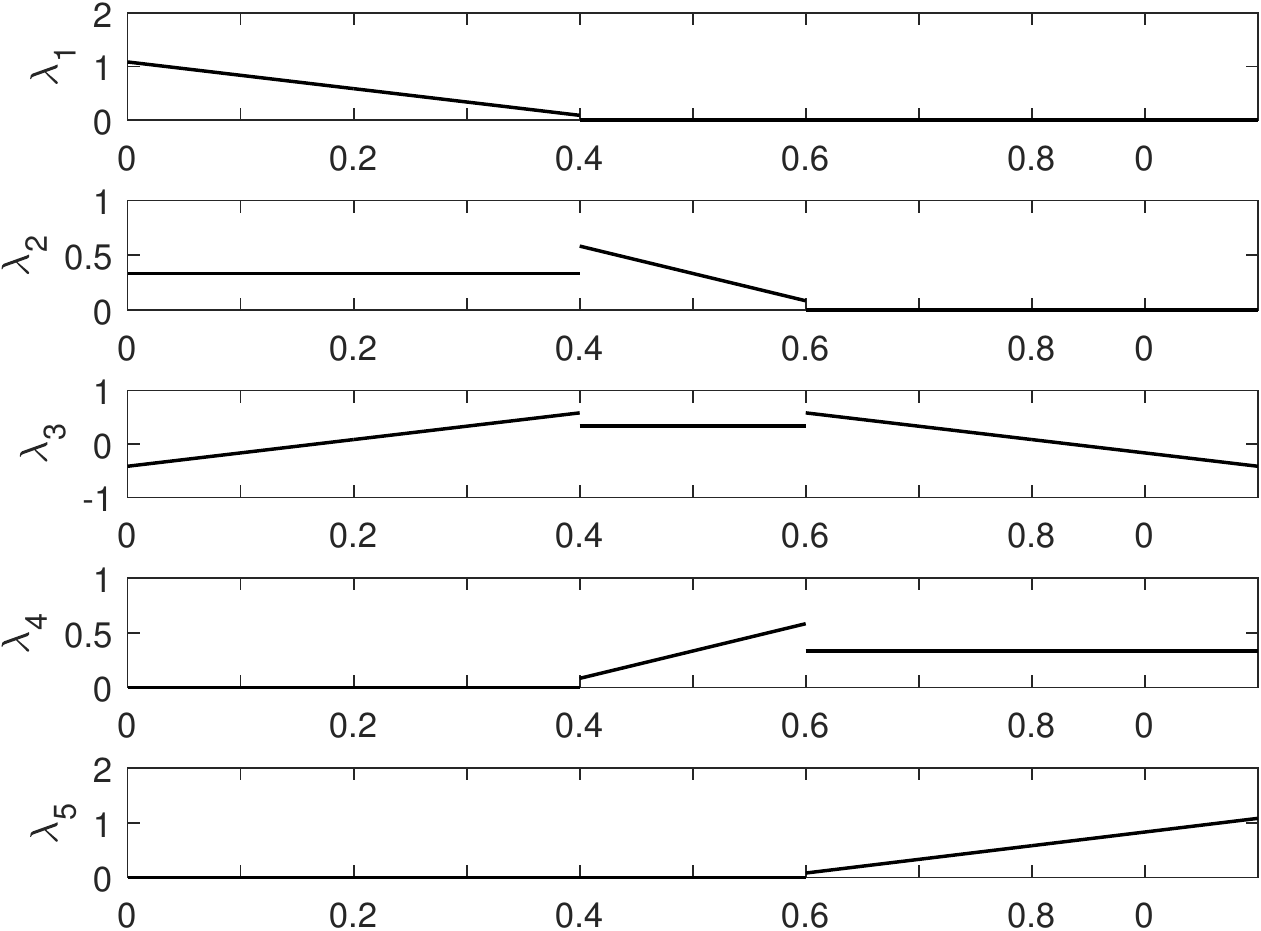}
  \end{tabular}
  \caption{The basis functions.}
  \label{fig:1d_basis}
\end{figure}
Thus we store $(A^TA)^{-1}A^T$ for each element to represent the basis
functions and it is the same when we deal with the high
dimensional problem.
}
\subsection{2D Smooth Solutions}
We firstly study the convergence rate for smooth solutions of
two-dimensional problems.

\begin{table}[htp]
  \centering
  \caption{uniform $\# S(K)$ for 2D smooth solutions}
  \label{tab:patchnumber2d}
  \scalebox{1.00}{
  \begin{tabular}{|l|l|l|l|l|l|l|}
    \hline \multicolumn{2}{|l|}{polynomial degree $m$}& 2 & 3 & 4 & 5
    & 6\\ \hline \multirow{3}{*}{$\# S(K)$} & Example 1 & 9 & 15 & 22
    & 29 & 38\\ \cline{2-7} & Example 2 & 9 & 16 & 23 & 32 &
    45\\ \cline{2-7} & Example 3 & 9 & 20 & 28 & 38 & 49\\ \hline
  \end{tabular}
  }
\end{table}
\textbf{Example 1.} Consider the biharmonic problem on the domain
$\Omega=(0,1)^2$ with Dirichlet boundary condition. The exact solution
is taken as
\begin{equation}
  u(x,y)=\sin^2(\pi x)\sin^2(\pi y),\quad (x,y)\in\Omega,
  \label{eq:ex1}
\end{equation}
and $g_D$, $g_N$ and the source term $f$ are chosen accordingly. The
polynomial degree $m$ is taken by $m=2,\cdots,6$. And the domain
$\Omega$ is partitioned into several regular disjoint elements for
each mesh size $h$; see Figure~\ref{fig:triangulation}, where $h=1/10,
1/20,1/40$ and $h=1/80$. We choose $\# S(K)$ uniformly for all
elements as in the second row of Table~\ref{tab:patchnumber2d}.

In Figure~\ref{fig:test1error}, we present the errors measured in both
the DG norm and the $L^2$ norm. It is clear that the convergence rate
in the DG norm is $m-1$ for fixed $m$. The convergence rate in the
L$^2$ norm is $m+1$ for $m\geq3$, which converges quadratically when
$m=2$. Such convergence rate is consistent with the theoretical
prediction in Theorem~\ref{th:errorbound}.
\begin{figure}[!htp]
  \centering
  \includegraphics[width=0.4\textwidth]{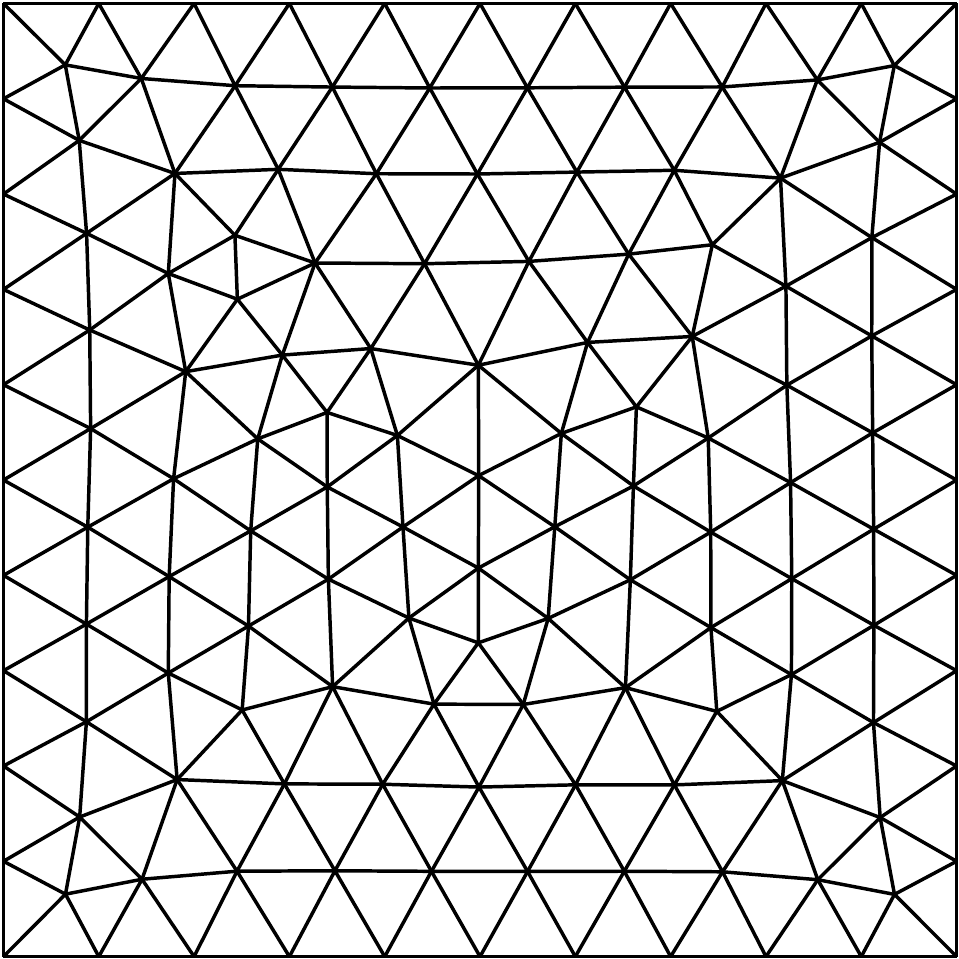}
  \hspace{25pt}
  \includegraphics[width=0.4\textwidth]{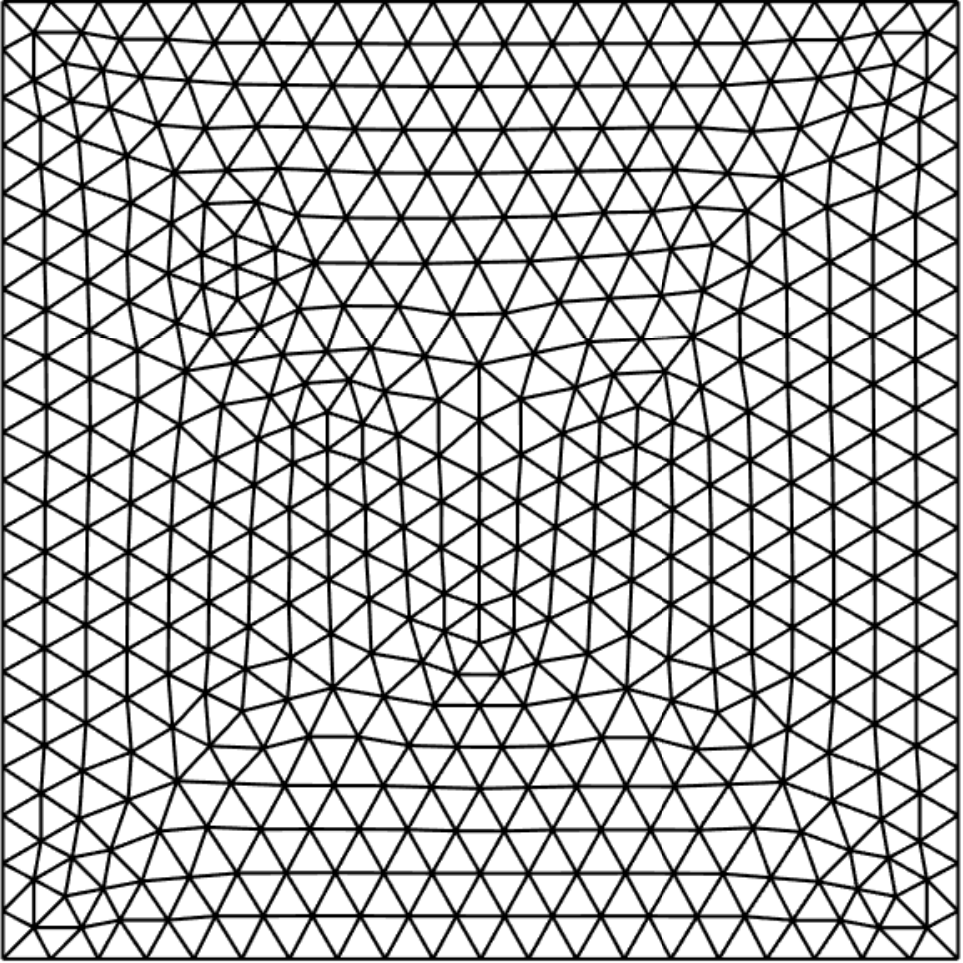}
  \caption{The triangular meshes for Example 1.}
  \label{fig:triangulation}
\end{figure}

\begin{figure}
  \centering
  \includegraphics[width=0.48\textwidth]{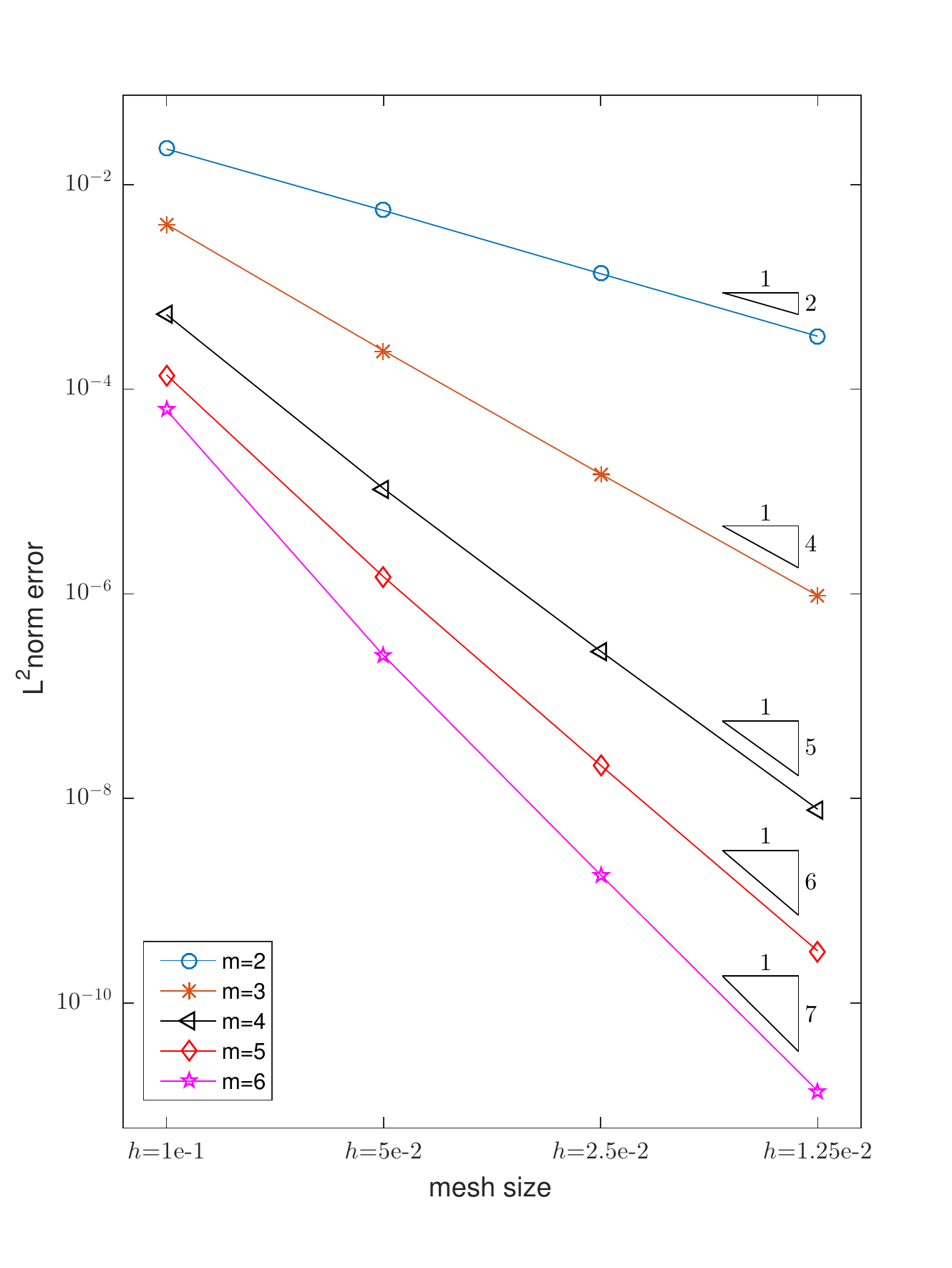}
  \includegraphics[width=0.48\textwidth]{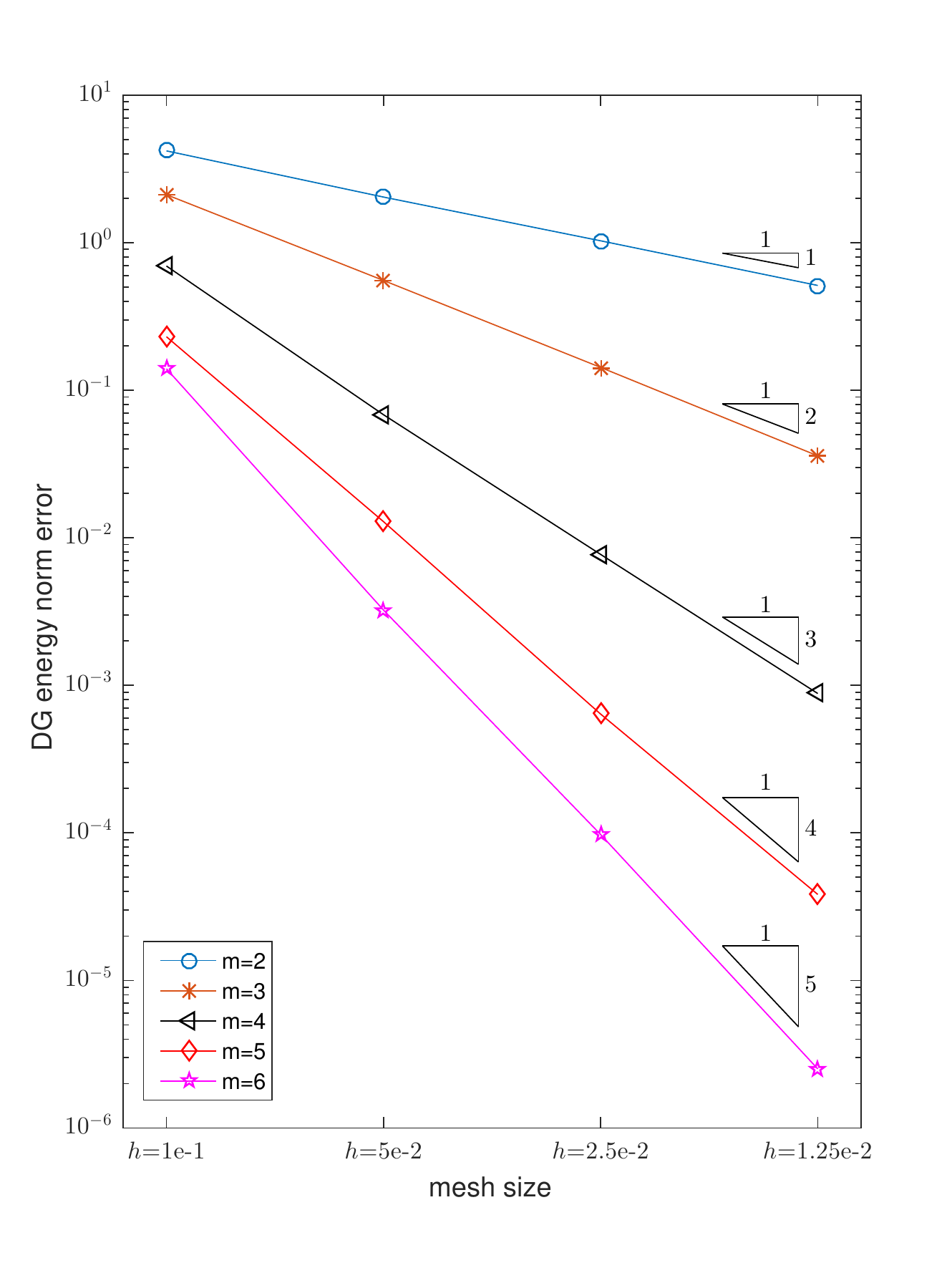}
  \caption{Examples 1: The convergence rate in the L$^2$ norm (left)
    and the DG norm (right) for different $m$ on triangular meshes.}
  \label{fig:test1error}
\end{figure}

\textbf{Example 2.} As we emphasize before, the local least-squares
problem~\eqref{eq:lsproblem} is independent of the element
geometry. In this example, we use the mesh generator
in~\cite{talischi2012polymesher} to obtain a series of polygonal
meshes from the Voronoi diagram of a given set of their reflections;
see Figure~\ref{fig:polymesh}. We also take~\eqref{eq:ex1} as the
exact solution. The number $\# S(K)$ is chosen as in the third row of
Table~\ref{tab:patchnumber2d}.

For each fixed $m$, we present the errors in both the DG norm and the
$L^2$ norm against the number of elements; see
Figure~\ref{fig:test2error}. It is clear that the numerical results
still agree well with the theoretical results.
%
\begin{figure}
  \centering
  \includegraphics[width=0.4\textwidth]{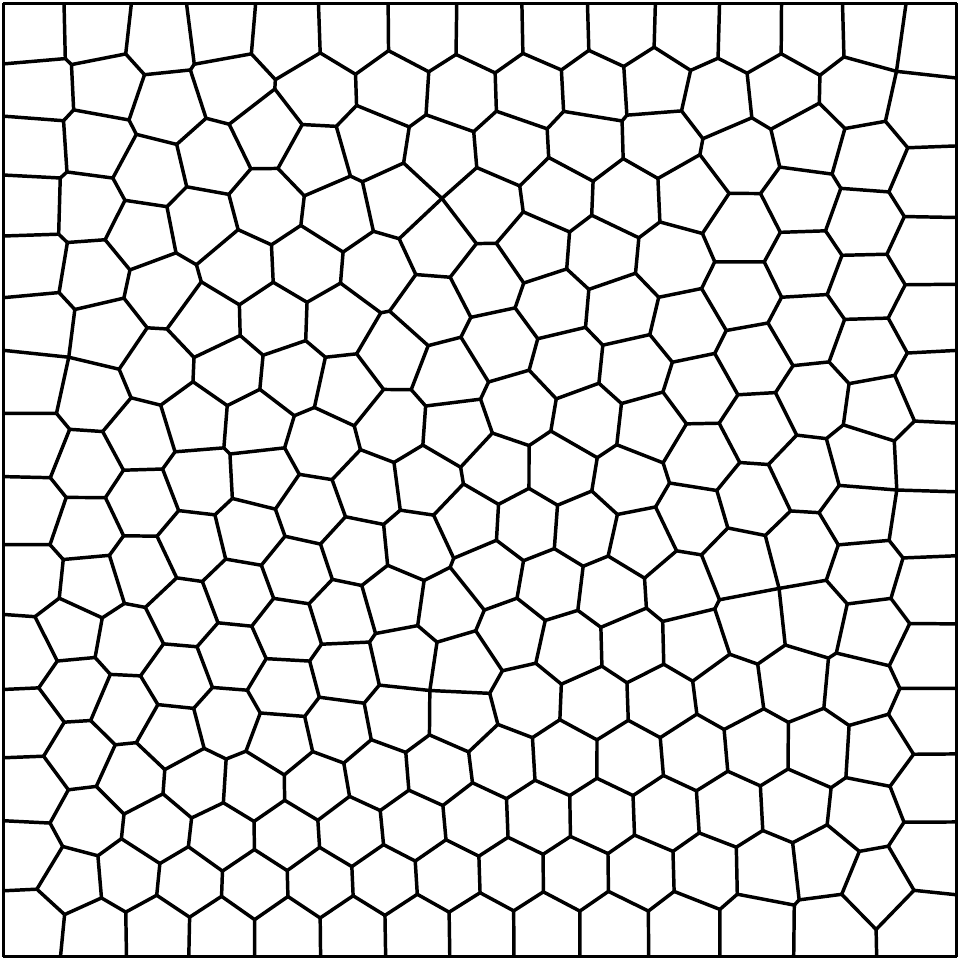}
  \hspace{25pt}
  \includegraphics[width=0.4\textwidth]{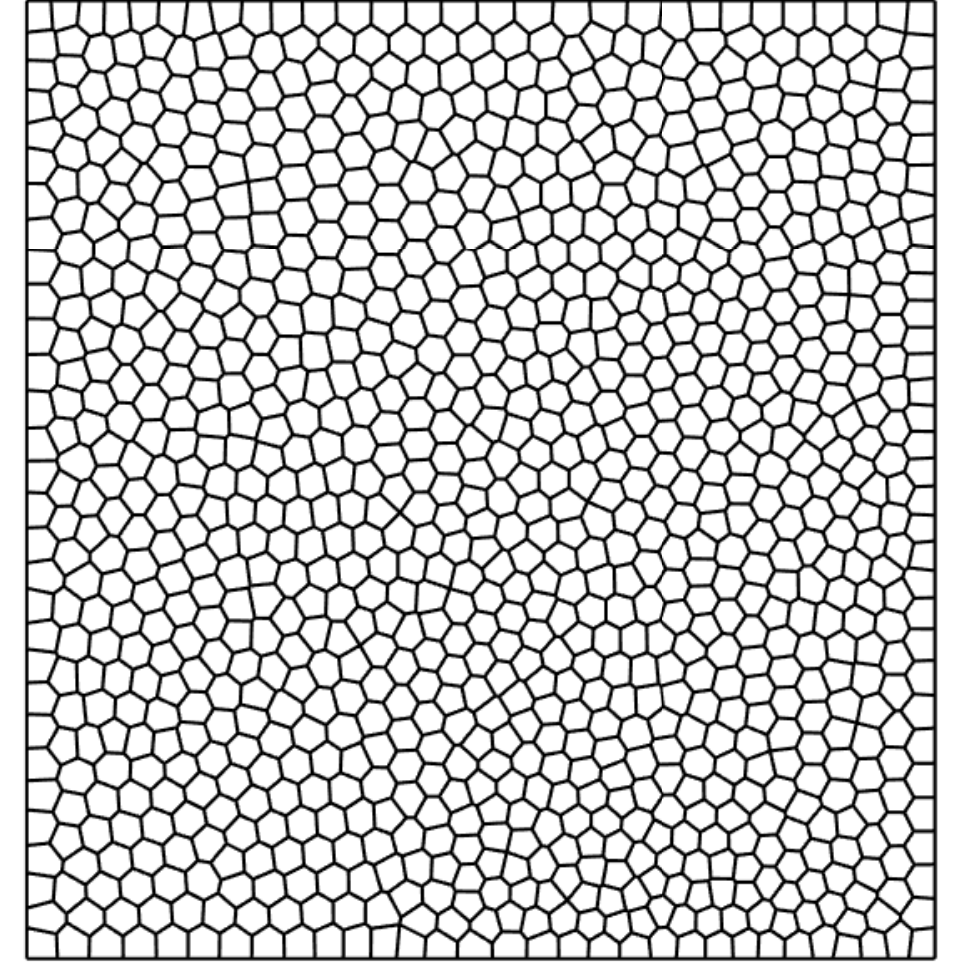}
  \caption{The Voronoi meshes for Example 2.}
  \label{fig:polymesh}
\end{figure}

\begin{figure}
  \centering
  \includegraphics[width=0.48\textwidth]{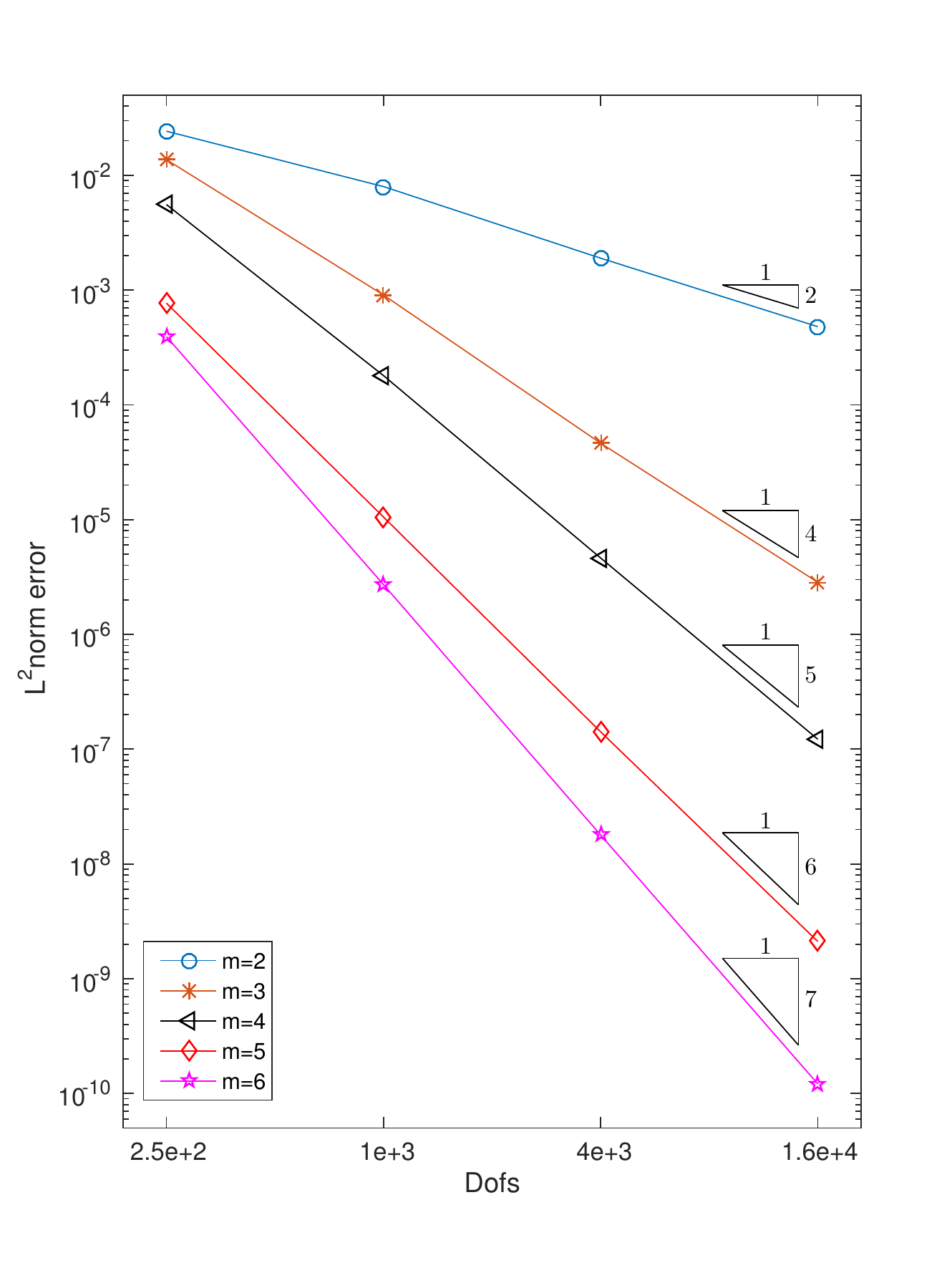}
  \includegraphics[width=0.48\textwidth]{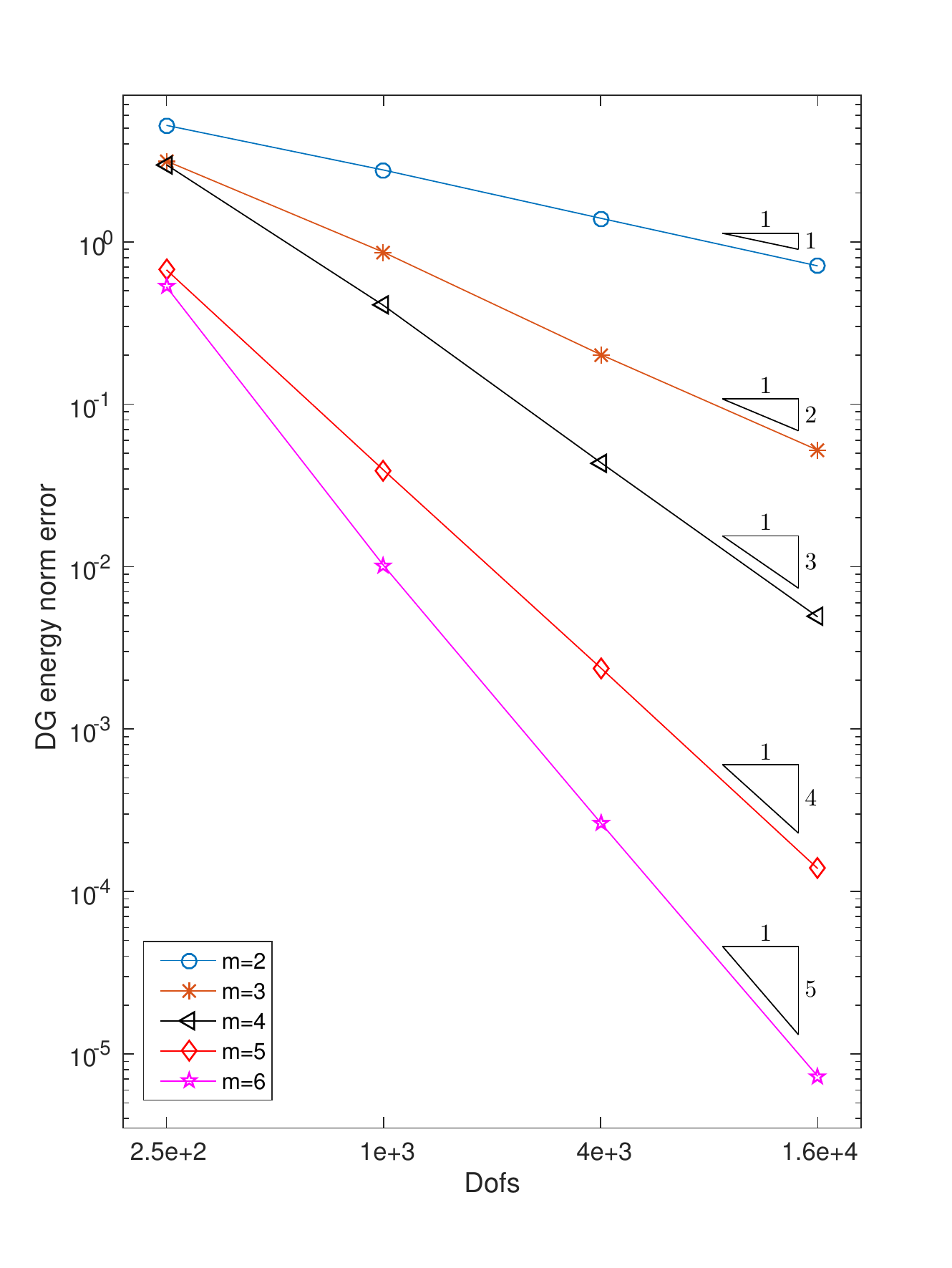}
  \caption{Examples 2: The convergence rate in the L$^2$ norm (left)
    and the DG energy norm (right) for different $m$ on Voronoi
    meshes.}
  \label{fig:test2error}
\end{figure}

\textbf{Example 3.} In this test, we consider the biharmonic problem
on $\Omega=(0,1)^2$ with the following boundary
condition~\cite{brenner2015C0}:
\[
u=\Delta u=0,\quad x\in\partial\Omega,
\]
which is related to the bending of a simply supported plate. We take
\(u(x,y)=\sin(2\pi x)\sin(2\pi y)\) as the exact solution. The mesh
consists of a mixing of triangular and quadrilateral elements, which
are generated by~\emph{gmsh}\cite{geuzaine2009gmsh}; see
Figure~\ref{fig:mixed}. The mesh size $h$ varies equally from $1/10$
to $h=1/80$. We take $\# S(K)$ as in the last row of
Table~\ref{tab:patchnumber2d}. The results in
Figure~\ref{fig:mixederror} show the convergence rate for different
$m$, which are also consistent with the theoretical results.
\begin{figure}
  \centering
  \includegraphics[width=0.4\textwidth]{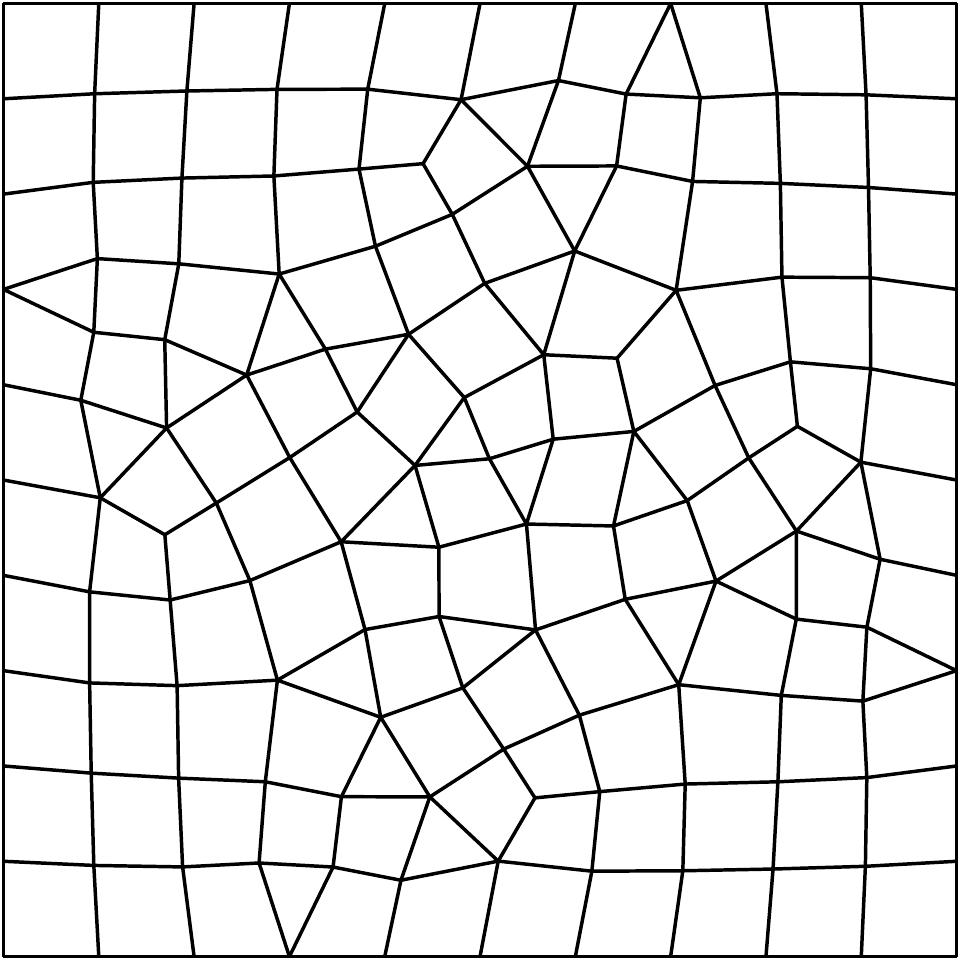}
  \hspace{25pt}
  \includegraphics[width=0.4\textwidth]{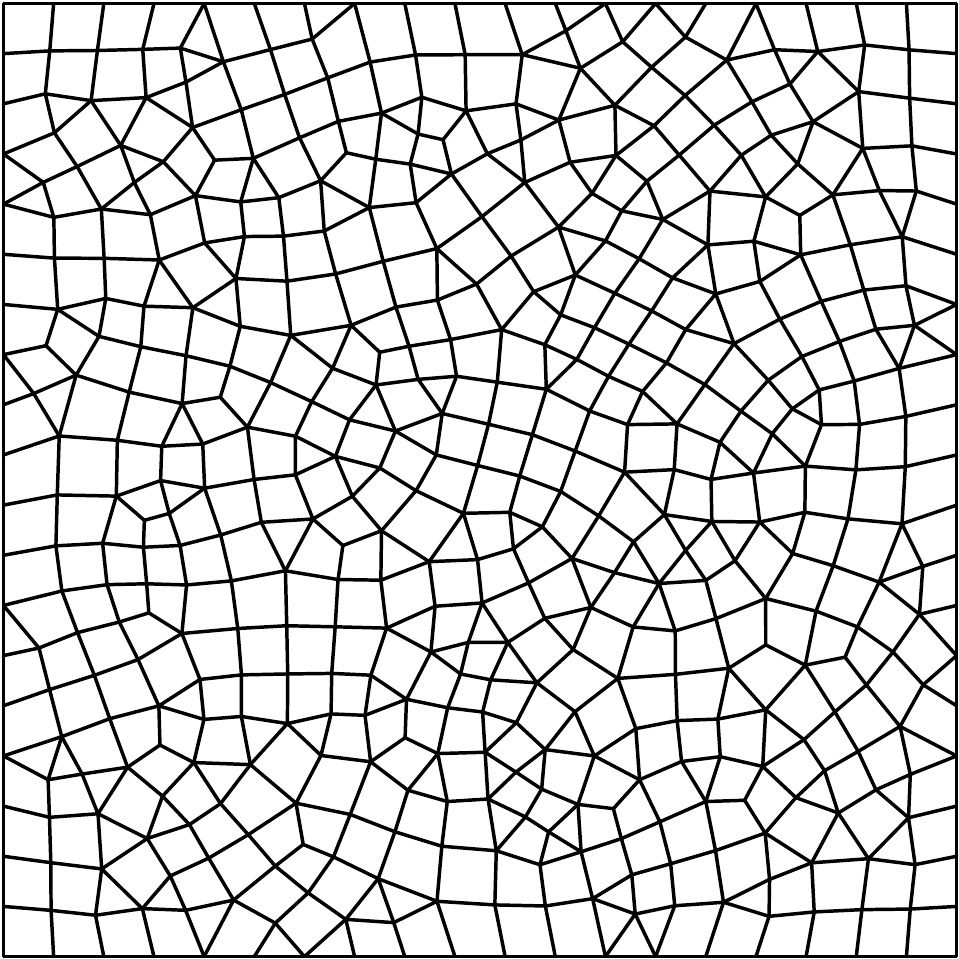}
  \caption{The mixed meshes for Example 3.}
  \label{fig:mixed}
\end{figure}
\begin{figure}
  \centering
  \includegraphics[width=0.4\textwidth]{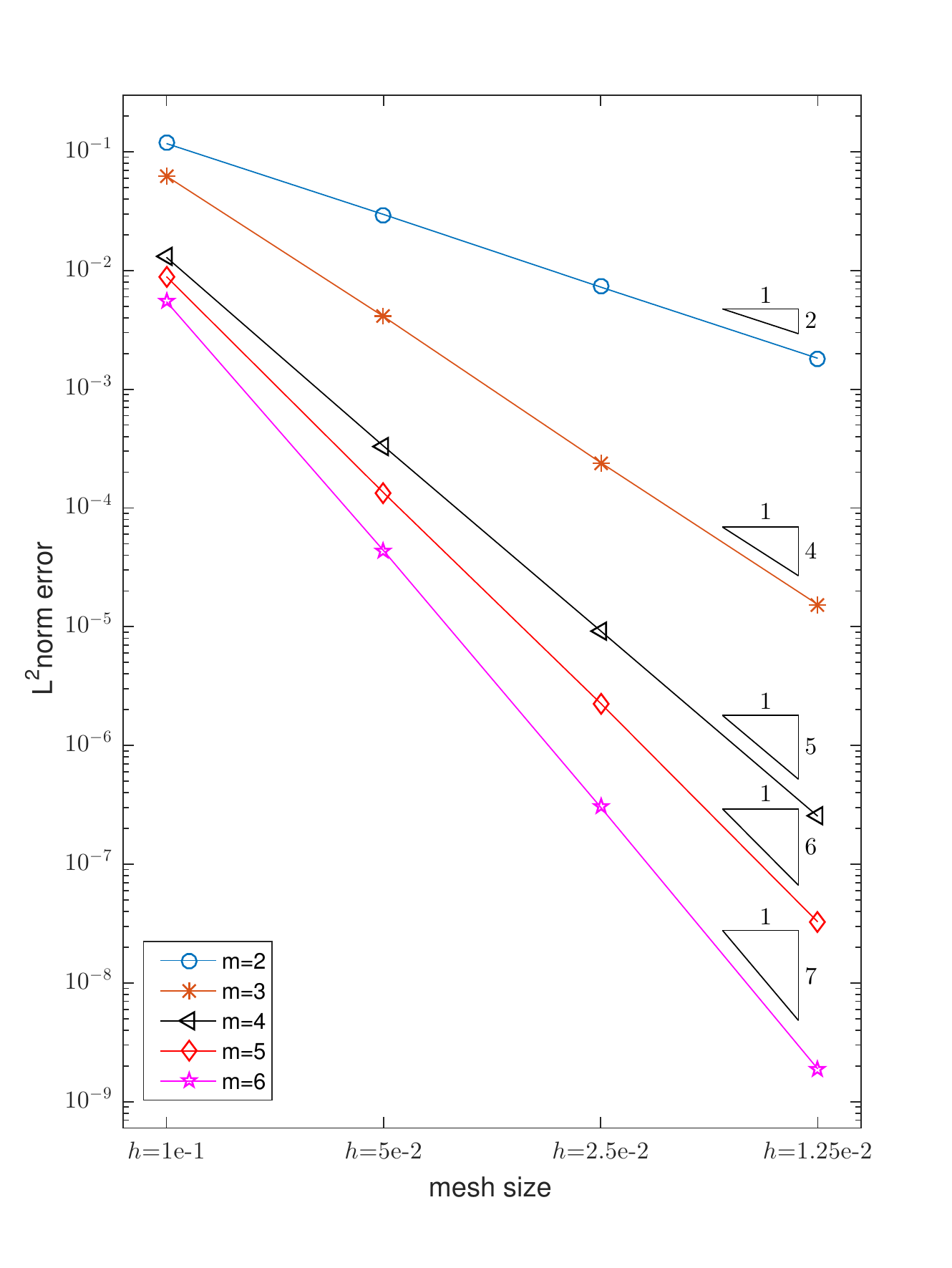}
  \includegraphics[width=0.4\textwidth]{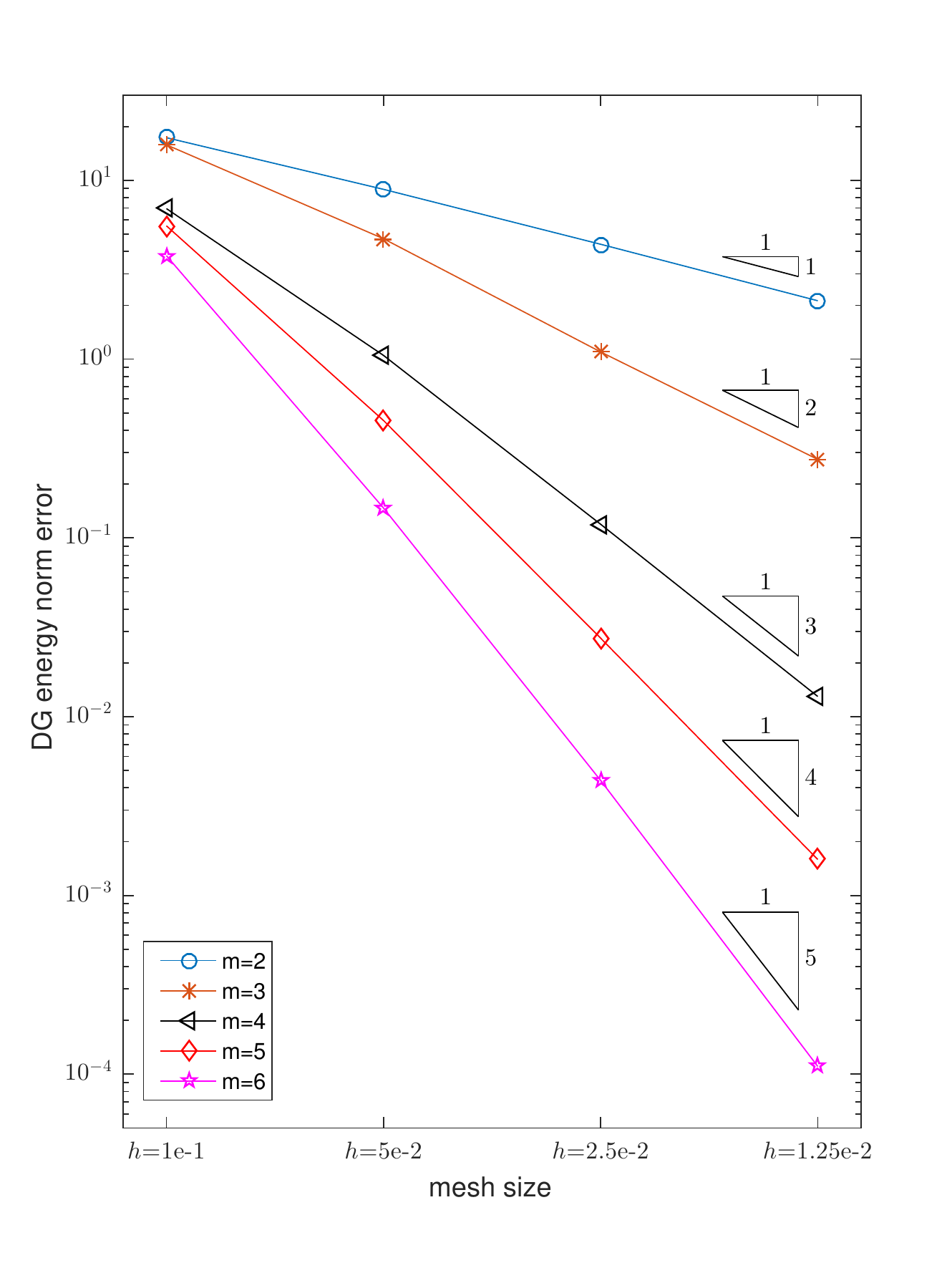}
  \caption{Examples 3: The convergence rate in the L$^2$ norm (left)
    and the DG energy norm (right) for different $m$ on mixed meshes.}
  \label{fig:mixederror}
\end{figure}

{\textbf{Example 4.} In this test, we compare the $C^0$IPG, IPDG and
the proposed method by solving the biharmonic problem in the domain
$\Omega = (0,1)^2$ with the same exact solution as for Example 1. The
meshes used for this case are obtained by successively refining an
initial mesh with $h=0.2$. We measure the errors in both the broken
$H^2$ norm and the $L^2$ norm. Figure~\ref{fig:compare1} and
Figure~\ref{fig:compare2} show the performance of three methods by
using spaces of polynomials of degree 2 and 3, respectively. We plot
the errors in both norms against the number of degrees of freedom. It
is clear that the proposed method behaves slightly better than other
methods.
\begin{figure}[htp]
  \centering
  \includegraphics[width=0.4\textwidth]{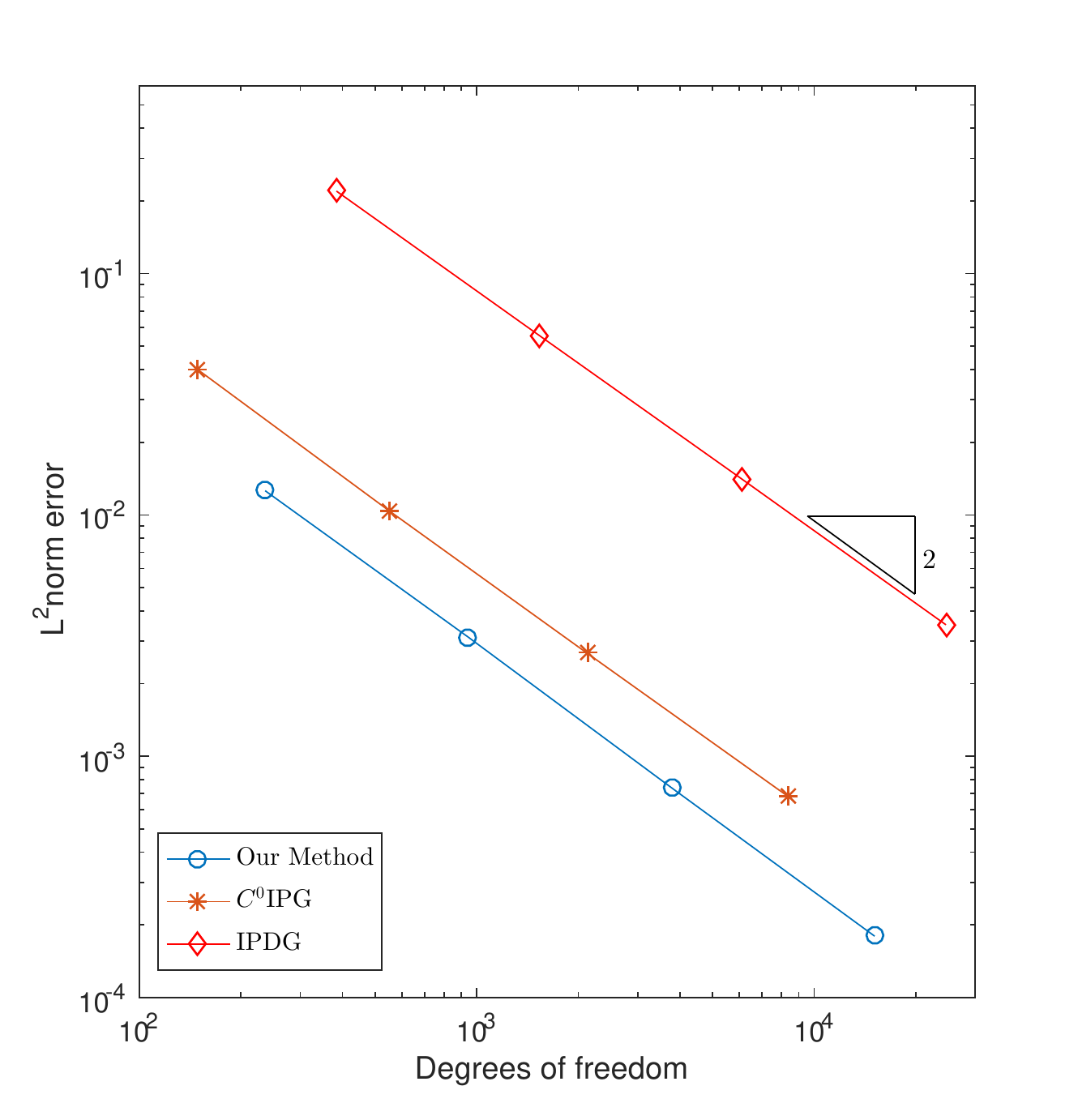}
  \hspace{0.1\textwidth}
  \includegraphics[width=0.4\textwidth]{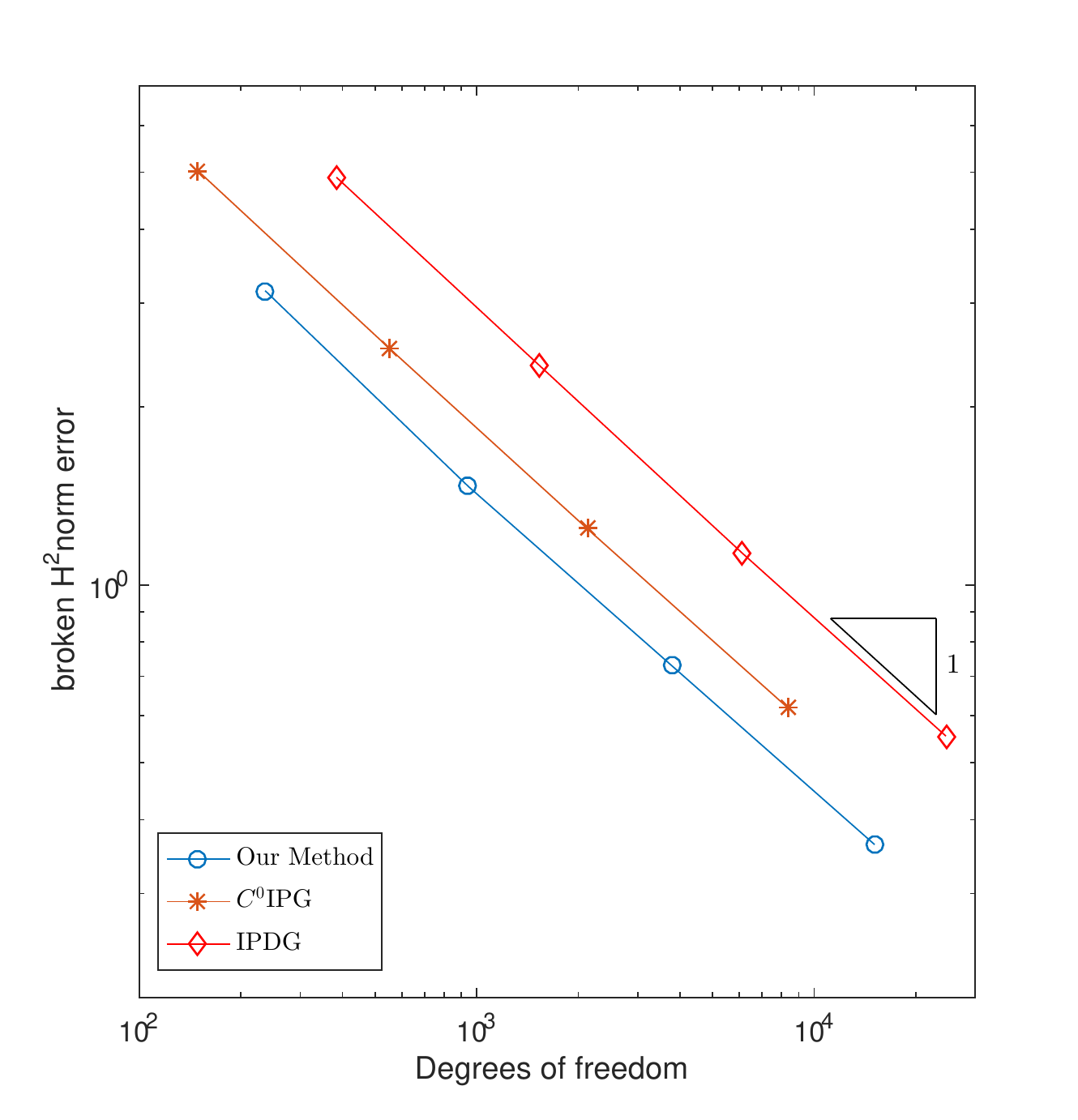}
  \caption{The error in $\| \cdot \|_{L^2(\Omega)}$ (left)/ $\|
    \cdot\|_{H^2(\Omega, \mathcal T_h)}$ (right) for three methods by
    using second order polynomials.  }
  \label{fig:compare1}
\end{figure}

\begin{figure}[htp]
  \centering
  \includegraphics[width=0.4\textwidth]{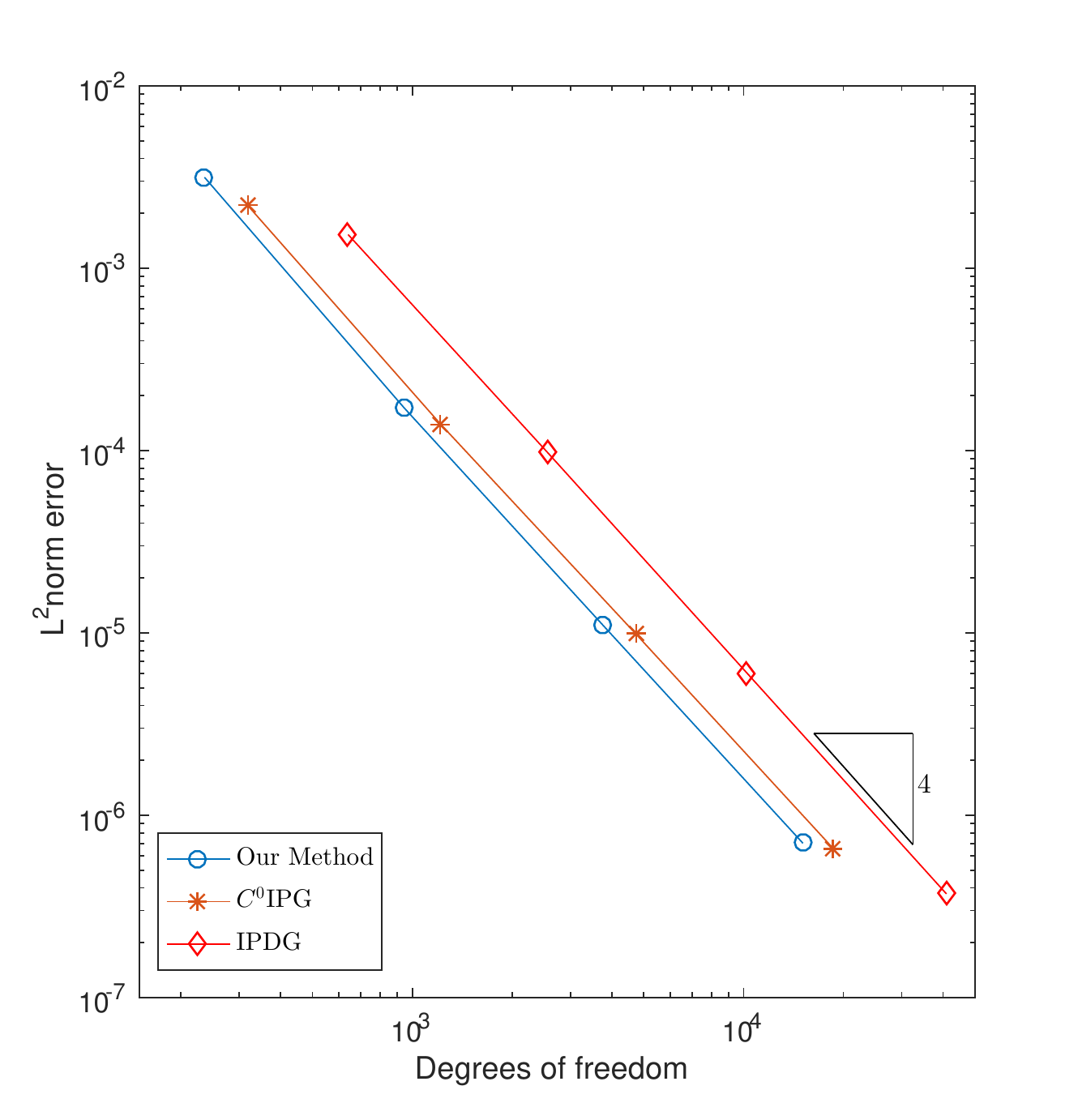}
  \hspace{0.1\textwidth}
  \includegraphics[width=0.4\textwidth]{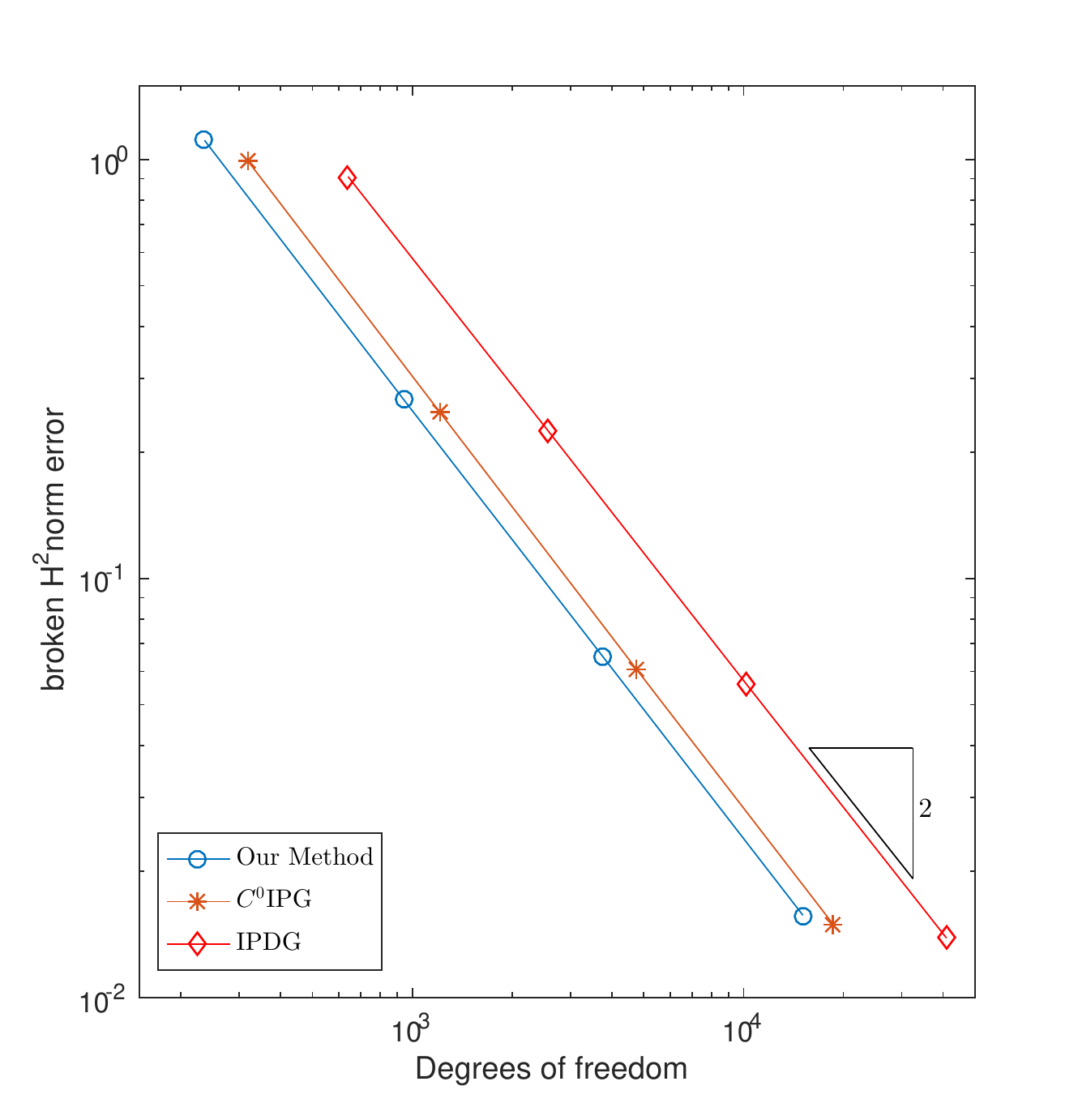}
  \caption{The error in $\| \cdot \|_{L^2(\Omega)}$ (left)/ $\|
    \cdot\|_{H^2(\Omega, \mathcal T_h)}$ (right) for three methods by
    using third order polynomials.  }
  \label{fig:compare2}
\end{figure}
}
\subsection{L-shaped domain with known exact solution}
In this example, we study the performance of the proposed method with
the problem with a corner singularity. Let $\Omega$ be the L-shaped
domain $(-1,1)^2\backslash [0,1)\times(-1,0]$ and we use triangular
meshes; see
Figure~\ref{fig:Lshape}. Following~\cite{georgoulis2008discontinuous},
we let
\[
u(r,\theta)=r^{5/3}\sin(5\theta/3)
\]
in polar coordinate and impose Dirichlet boundary condition. At the
corner $(0,0)$ the exact solution contains a singularity which
indicates $u$ only belongs to $H^{8/3-\epsilon}(\Omega)$ for
$\epsilon>0$. The number $\# S(K)$ is chosen as in the second row of
Table~\ref{tab:patchnumber2d}.

In Table~\ref{tab:Lshapeerror}, we list the error measured in the DG
norm and the $L^2$ norm against the mesh size for $m=2,3,4$.  Here we
observe that the error in $L^2$ norm decreases at the rate
$\mc{O}(h^{1.2})$ while the error in DG norm decreases at the rate
$\mc{O}(h^{2/3})$. It seems the convergence rates agree with that
in~\cite{georgoulis2008discontinuous}.
\begin{figure}
  \centering
  \includegraphics[width=0.4\textwidth]{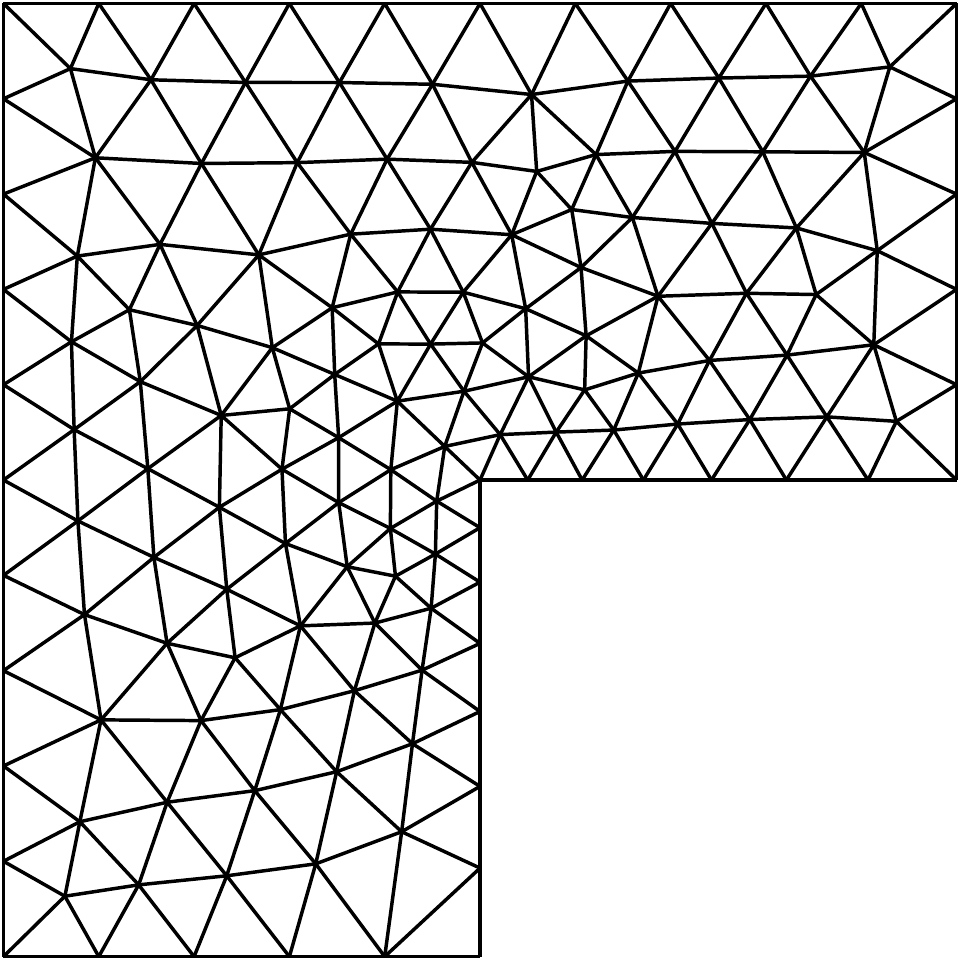}
  \hspace{25pt}
  \includegraphics[width=0.4\textwidth]{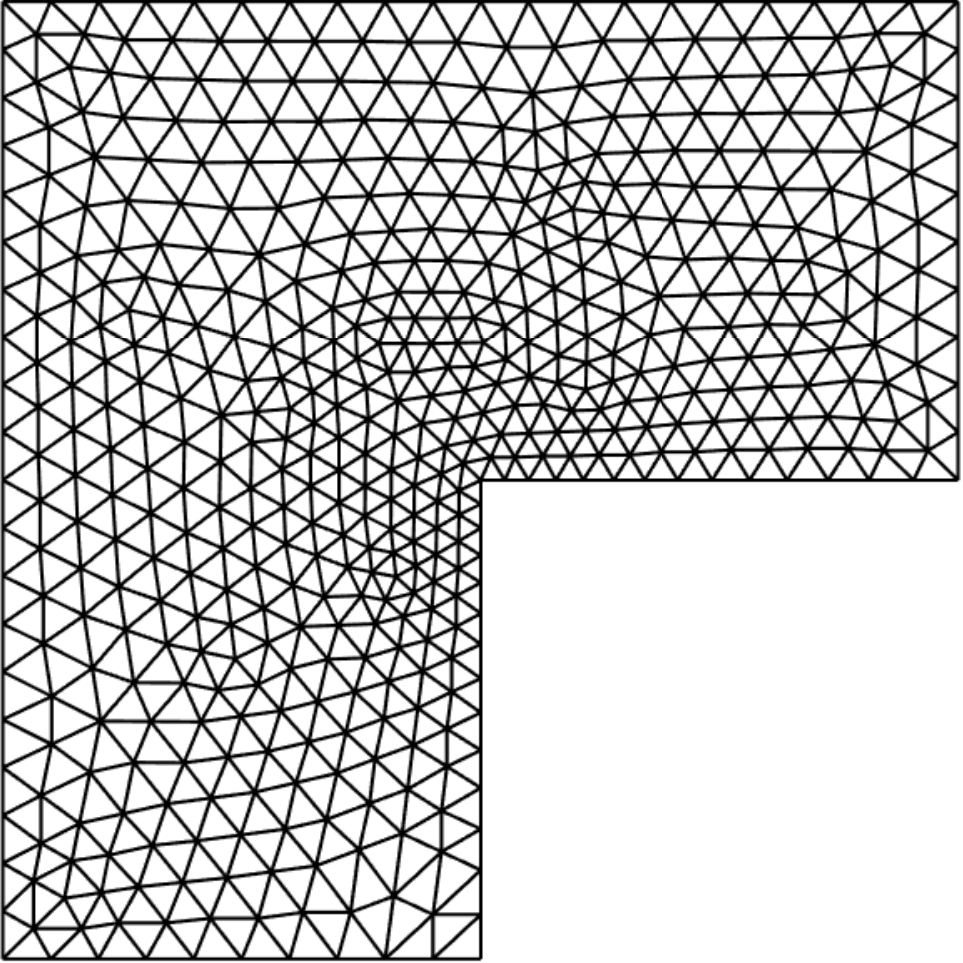}
  \caption{The triangular meshes of L-shaped domain}
  \label{fig:Lshape}
\end{figure}

\begin{table}[]
  \centering
  \caption{Convergence rates of L-shaped domain example}
  \label{tab:Lshapeerror}
  \scalebox{0.72}{
  \begin{tabular}{|l|l|l|l|l|l|l|l|l|l|l|}
    \hline \multirow{3}{*}{$m$} & \multirow{3}{*}{Norm} & Dofs & Dofs
    & \multirow{3}{*}{Order} &Dofs & \multirow{3}{*}{Order} &Dofs &
    \multirow{3}{*}{Order} & Dofs & \multirow{3}{*}{Order}
    \\ \cline{3-4} \cline{6-6} \cline{8-8} \cline{10-10} & & $250$ &
    $1000$ & & $4000$ & & $16000$ & &$64000$ & \\ \cline{3-4}
    \cline{6-6} \cline{8-8} \cline{10-10}& & Error & Error & & Error &
    & Error & & Error & \\ \hline \multirow{2}{*}{2} &
    $\|\cdot\|_{L^2(\Omega)}$ & $1.38\mathrm e$-3 & $6.15\mathrm e$-4
    & 1.17 & $2.68\mathrm e$-4 & 1.23 & $1.17\mathrm e$-4 & 1.20 &
    $5.13\mathrm e$-5 & 1.19 \\ \cline{2-11} & $\|\cdot\|_{h}$ &
    $3.35\mathrm e$-1 & $2.03\mathrm e$-1 & 0.72 & $1.23\mathrm e$-1 &
    0.72 & $7.63\mathrm e$-2 & 0.69 & $4.73\mathrm e$-2 & 0.68
    \\ \hline \multirow{2}{*}{3} & $\|\cdot\|_{L^2(\Omega)}$ &
    $8.58\mathrm e$-4 & $3.11\mathrm e$-4 & 1.47 & $1.22\mathrm e$-4 &
    1.35 & $5.33\mathrm e$-5 & 1.19 & $2.99\mathrm e$-5 & 1.21
    \\ \cline{2-11} & $\|\cdot\|_{h}$ & $2.42\mathrm e$-1 &
    $1.31\mathrm e$-1 & 0.88 & $8.43\mathrm e$-1 & 0.64 & $5.33\mathrm
    e$-2 & 0.66 & $3.37\mathrm e$-2 & 0.66 \\ \hline
    \multirow{2}{*}{4} & $\|\cdot\|_{L^2(\Omega)}$ & $1.08\mathrm e$-3
    & $3.19\mathrm e$-4 & 1.76 & $1.11\mathrm e$-4 & 1.52 &
    $4.56\mathrm e$-5 & 1.28 & $1.95\mathrm e$-5 & 1.22
    \\ \cline{2-11} & $\|\cdot\|_{h}$ & $3.43\mathrm e$-1 &
    $1.76\mathrm e$-1 & 0.96 & $1.08\mathrm e$-1 & 0.71 & $6.78\mathrm
    e$-2 & 0.67 & $4.25\mathrm e$-2 & 0.67 \\ \hline
  \end{tabular}}
\end{table}
\subsection{3D Smooth Solution}
In this example, we solve a three-dimensional biharmonic problem on a
unit cube $\Omega=(0,1)^3$.  The domain is partitioned into
tetrahedral meshes with mesh size $h=1/4, 1/8, 1/16$ and $h=1/32$ by
\emph{gmsh}. The exact solution is chosen as $u(x,y,z)=\sin^2(\pi
x)\sin^2(\pi y)\sin^2(\pi z)$ and the function $g_D, g_N$ and $f$ are
taken suitably. We build the reconstruction operator with different
polynomial degrees $m=2,3,4$.  The number $\# S(K)$ is listed in
Table~\ref{tab:patchnumber3D}.
\begin{table}[]
  \centering
  \caption{Uniform $\# S(K)$ for 3D smooth solution}
  \label{tab:patchnumber3D}
  \scalebox{1.00}{
  \begin{tabular}{|c|l|l|l|}
    \hline polynomial degree $m$ & 2 & 3 & 4\\ \hline all $\# S(K)$ &
    21 & 40 & 62 \\ \hline
  \end{tabular}
  }
\end{table}

The numerical results are presented in Table~\ref{tab:3derror}. The
convergence rates also agree with the theoretical prediction.
\begin{table}[]
  \centering
  \caption{Convergence rates of 3D example}
  \label{tab:3derror}
  \scalebox{0.85}{
  \begin{tabular}{|l|l|l|l|l|l|l|l|l|}
    \hline \multirow{3}{*}{m} & \multirow{3}{*}{Norm} & Dofs & Dofs &
    \multirow{3}{*}{Order} & Dofs & \multirow{3}{*}{Order} & Dofs &
    \multirow{3}{*}{Order} \\ \cline{3-4} \cline{6-6} \cline{8-8} & &
    384 & 3072 & & 24576 & & 196608 & \\ \cline{3-4} \cline{6-6}
    \cline{8-8} & & Error & Error & & Error & & Error & \\ \hline
    \multirow{2}{*}{2} & $\|\cdot\|_{L^2(\Omega)}$ & $7.34\mathrm e$-2
    & $1.43\mathrm e$-2 & 2.36 & $3.34\mathrm e$-3 & 2.10 &
    $8.07\mathrm e$-3 & 2.05 \\ \cline{2-9} & $\|\cdot\|_{h}$ &
    $8.87\mathrm e$-0 & $4.59\mathrm e$-0 & 0.95 & $2.41\mathrm e$-0 &
    0.93 & $1.22\mathrm e$-0 & 0.98 \\ \hline \multirow{2}{*}{3} &
    $\|\cdot\|_{L^2(\Omega)}$ & $3.34\mathrm e$-2 & $3.76\mathrm e$-3
    & 3.15 & $2.50\mathrm e$-4 & 3.91 & $1.59\mathrm e$-5 & 3.97
    \\ \cline{2-9} & $\|\cdot\|_{h}$ & $6.89\mathrm e$-0 &
    $2.19\mathrm e$-0 & 1.66 & $5.86\mathrm e$-1 & 1.90 & $1.50\mathrm
    e$-1 & 1.97 \\ \hline \multirow{2}{*}{4} &
    $\|\cdot\|_{L^2(\Omega)}$ & $2.83\mathrm e$-2 & $8.96\mathrm e$-4
    & 4.98 & $2.43\mathrm e$-5 & 5.19 & $7.11\mathrm e$-7 & 5.09
    \\ \cline{2-9} & $\|\cdot\|_{h}$ & $5.09\mathrm e$-0 &
    $1.03\mathrm e$-0 & 2.31 & $1.42\mathrm e$-1 & 2.86 & $1.68\mathrm
    e$-2 & 3.08 \\ \hline
  \end{tabular}}
\end{table}
\section{Conclusion}\label{sec:conclusion}
We propose a new discontinuous Galerkin method to solve the biharmonic
boundary value problem.  A novelty of the method is a new
discontinuous polynomial space that is reconstructed by solving local
least-squares.  The optimal error estimates in both the DG energy norm
and the $L^2$ norm are proved, which are confirmed by a series of
numerical examples with different complexity.

\bibliographystyle{abbrv}
\bibliography{ref}

\end{document}